\documentclass[11pt,a4paper]{article}
 \usepackage{euscript}
\usepackage{amsmath}
\usepackage{graphicx}
\usepackage{amsthm}
\usepackage{amssymb}
\usepackage{epsfig}
\usepackage{url}
\usepackage{amsmath,amscd}
\theoremstyle{theorem}
\newtheorem{theorem}{Theorem}[section]
\newtheorem{corollary}[theorem]{Corollary}

\newtheorem{lemma}[theorem]{Lemma}
\newtheorem{proposition}[theorem]{Proposition}

\newtheorem{example}{Example}[section]
\newtheorem{remark}{Remark}[section]
\numberwithin{equation}{section}
 
\newcommand{\s}{\sigma}
\title{Algebraically closed $\s$-fields} 
\author{Masood Aryapoor\\
\tiny{\textit{Division of Mathematics and Physics}}\\
\tiny{\textit{M\"{a}lardalen  University}}\\
\tiny{\textit{Hamngatan 15, 632 17, Eskilstuna, 
Sweden
}}

}
 \date{}
\begin{document}
 \maketitle
\begin{abstract}
\noindent
The concept of a skew root of a skew polynomial is used to introduce notions of algebraic closedness for $\s$-fields, that is, a field equipped with an endomorphism. It is shown that every $\s$-field can be embedded in algebraically closed $\s$-fields of different types.     \\
%\textit{Keywords:Skew polynomial, Skew root, $\s$-field.} 
\end{abstract}
%%%%%%%%%%%%%%%%%%%%%%%%%%%%%%%%%%%%%%%%%%%%%%%%%%%%%%%%%%%%%%%%
%%%%%%%%%%%%%%%%%%%%%%%%%%%%%%%%%%%%%%%%%%%%%%%%%%%%%%%%%%%%%%%%
%%%%%%%%%%%%%%%%%%%%%%%%%%%%%%%%%%%%%%%%%%%%%%%%%%%%%%%%%%%%%%%%
%%%%%%%%%%%%%%%%%%%%%%%%%%%%%%%%%%%%%%%%%%%%%%%%%%%%%%%%%%%%%%%%
\begin{section}{Introduction} 
Skew polynomials were first studied by Ore \cite{ore1933theory}. Since Ore's article was published in 1933, skew polynomials have been the subject of study by many  authors. In \cite{lam1985general},  Lam introduced the notion of a skew root of a skew polynomial in order to study a generalization of the Vandermonde matrix (see also \cite{LamLeroy1988}). This notion allows us to define a concept of algebraicity which has been used by Lam and Leroy to formulate and solve some problems regarding skew polynomials, see  \cite{lam2004wedderburn,lam1994hilbert,lamleroy1988algebraic}.  The main goal of this paper is to study this concept of  algebraicity in the context of field extensions. 

The paper is organized as follows. Section 2 deals with preliminaries and elementary problems such as interpolation for skew polynomials and Vieta's formula for skew polynomials. In Section 3, we study field extensions from the perspective of skew polynomials. It is shown that for any nonconstant skew polynomial, there is an extension field in which the given skew polynomial has a skew root. Section 4 deals with the main subject of this paper, namely, algebraically closed $\s$-fields. It turns out that there are different notions of  "algebraic closedness" in the context of skew roots, some of which are introduced in this section, and moreover, it is shown that they exit.  Finally, in Section 5, we discuss the notion of a partial fraction decomposition for skew polynomials.  It is shown that an arbitrary  "skew rational  functions" over an algebraically closed $\s$-field has a relatively simple partial fraction decomposition.  
       
We conclude the introduction with some remarks. In this paper, all rings are assumed  to be associative and contain an identity element. In general, a skew polynomial is an element of a skew polynomial ring $R[T;\sigma,\delta]$, where $R$ is a (not necessarily commutative) ring, $\sigma:R\to R$ is an endomorphism
 and $\delta:R\to R$ is a $\s$-derivation. However, in this paper, we work exclusively with skew polynomial rings $R[T;\sigma,\delta]$ for which $R$ is a commutative field and $\delta$ is the zero map. Given a set $X$, we denote the identity map of $X$ by $1_X$. For a subset $X$ of a ring $R$, the notation $X^*$ stands for $X\setminus \{0\}$.

\end{section} 
%%%%%%%%%%%%%%%%%%%%%%%%%%%%%%%%%%%%%%%%%%%%%%%%%%%%%%%%%%%%%%%%
%%%%%%%%%%%%%%%%%%%%%%%%%%%%%%%%%%%%%%%%%%%%%%%%%%%%%%%%%%%%%%%%

%%%%%%%%%%%%%%%%%%%%%%%%%%%%%%%%%%%%%%%%%%%%%%%%%%%%%%%%%%%%%%%%
%%%%%%%%%%%%%%%%%%%%%%%%%%%%%%%%%%%%%%%%%%%%%%%%%%%%%%%%%%%%%%%%
%%%%%%%%%%%%%%%%%%%%%%%%%%%%%%%%%%%%%%%%%%%%%%%%%%%%%%%%%%%%%%%%
%%%%%%%%%%%%%%%%%%%%%%%%%%%%%%%%%%%%%%%%%%%%%%%%%%%%%%%%%%%%%%%%

\begin{section}{Skew polynomials and skew roots}
	This section is devoted to  preliminaries and some elementary results regarding skew polynomials and skew roots.  We begin by reviewing basic facts about the ring of skew polynomials over a field, see \cite[Chapter 2]{cohn1995encyclopedia} for more details. Let $K$ be a (commutative) field and $\s:K\to K$ be an endomorphism. The pair $(K,\s)$ is called a $\s$-field. A (nonzero) skew polynomial $P(T)$ over $K$ is a "polynomial" of the form
	$$P(T)=a_0+a_1T+\cdots+a_n T^n,$$
	where $n\geq 0$ and $a_0,...,a_n\in K$ with $a_n\neq 0$. The number $n$, denoted by  $\deg P$, is called the degree of $P(T)$. We usually use the word "polynomial" instead of the expression "skew polynomial"  if there is no risk of confusion. The set of all skew polynomials over $K$ is denoted by $K[T;\sigma]$. This set has a ring structure which resembles the ring of ordinary polynomials. A crucial difference is that $Ta=\s(a)T$ for all $a\in K$. It is known that Euclidean algorithm for right division holds $K[T;\sigma]$, a fact which has a number of consequences: $K[T;\sigma]$ is a left principal ideal domain; $K[T;\sigma]$ is a left Ore domain whose field of fractions is denoted by    $K(T;\sigma)$. Elements of $K(T;\sigma)$ can be expressed as $P(T)^{-1}Q(T)$ where $0\neq P(T), Q(T)\in K[T;\sigma]$. Besides the rings $K[T;\sigma]$ and $K(T;\sigma)$, one has the ring of "skew power series" $K[[T;\sigma]]$, and if $\s$ is an automorphism, the skew field $K((T;\sigma))$ of Laurent series over $(K,\s)$  whose elements are of the form 
	$$a_0+a_1T+\cdots+a_n T^n+\cdots,
	$$
	and 
	$$a_{-m}T^{-m}+\cdots+a_{-1}T^{-1}+a_0+a_1T+\cdots+a_n T^n+\cdots,
	$$
	respectively. We  have the following inclusions: 
	$$
	\begin{matrix}
		K[T;\sigma] & \subset & K[[T;\sigma]]\\
		\cap & & \cap\\
		K(T;\sigma)  & \subset & K((T;\sigma))
	\end{matrix}
	$$
	If $\s:K\to K$ is an automorphism, then $K[T;\sigma]$ is a principal ideal domain, that is, it is both a left and a right principal ideal domain. Moreover, any $P(T)\in K[T;\sigma]$  can be written as a "right polynomial":
	$$P(T)=b_0+T b_1+\cdots+T^n b_n,$$
	where  $b_0,...,b_n\in K$. In fact, if $P(T)$, as a left polynomial, has the expression 
	 $$P(T)=a_0+a_1T+\cdots+a_n T^n,$$ then $P(T)$, as a right polynomial, has the expression
	 $$P(T)=a_0+T\s^{-1}(a_1)+\cdots+ T^n\s^{-n}(a_n).$$
	 As an example of a ring of skew polynomials, we have the ring $\mathbb{C}[T;\, \bar \,\,]$ of skew polynomials over $\mathbb{C}$, where the endomorphism $\bar\,:\mathbb{C}\to \mathbb{C}$ is the familiar map sending a complex number to its complex conjugate.
%%%%%%%%%%%%%%%%%%%%%%%%%%%%%%%%%%%%%%%%%%%%%%%%%%%%%%%%%%%%%%%%% 
\begin{subsection}{Preliminaries}\label{(S)preliminaries}
	In this subsection, we review some known facts regarding roots of skew polynomials. The main references for this part are the papers  \cite{lam1985general, LamLeroy1988}. Throughout this subsection, $(K,\s)$ is a fixed $\s$-field. 
	
	An element $a\in K$ is called a \textit{$\s$-conjugate} of $b\in K$ \textit{over a subset} $X$ of $K$ if there exists $0\neq x\in X$ such that $a=\sigma(x)bx^{-1}$.  We drop the phrase "over $X$" when $X=K$. It is easy to see that the notion of $\s$-conjugacy over a subgroup $X$ of the multiplicative group $K^*$ is an equivalence relation in which case the $\s$-conjugacy class containing $a\in K$ is denoted by $C(a,X)$. When $X=K^*$, we write $C(a)$ instead of $C(a,X)$. 		
		
	The maps $N_i:K\to K$, where $i\geq 0$, are defined recursively as follows
	$$N_{i}(a)=\begin{cases}
		1& \text{ if }i=0,\\
		\sigma(N_{i-1}(a))\,a & \text{ if } i>0.\\
	\end{cases}
	$$
	Clearly, $N_i(a)=\s^{i-1}(a)\cdots \s(a)a$ for all $i\geq 1$. It is easy to verify the following identities:
	\begin{equation}\label{(Id)Ni+j}
		N_{i+j}(a)=\sigma^i(N_j(a))N_i(a), \text{ for all } a\in K \text{ and } i,j\geq 0,
	\end{equation}
	\begin{equation}\label{(Id)Niofconjugate}
		N_i(\sigma(a)ba^{-1})=\sigma^i(a)N_i(b)a^{-1}, \text{ for all } a\in K^*, b\in K \text{ and } i\geq 0.
	\end{equation}
	Note that $N_i(ab)=N_i(a)N_i(b)$ for all $a,b\in K$. Furthermore, if $c\in K$ is fixed by $\s$, then $N_i(c)=c^i$ for all $i\geq 0$.
 
	A polynomial $P(T)=\sum_{i=0}^{n}a_iT^i\in K[T;\s]$ can be regarded as a function $P:K\to K$ as follows:
	$$P(a):=\sum_{i=0}^{n}a_iN_i(a).$$
	Note that $P(c)=\sum_{i=0}^{n}a_ic^i$ for all $c$ in the fixed field of $\s$, that is, the field $\{c\in K\,|\, \s(c)=c\}$. The element $P(a)$ is called the value of $P$ at $a$. This definition is justified by the following lemma:
	\begin{lemma}\label{(L)DivisionP(a)}
	For every $P(T)\in K[T;\sigma]$ and $a\in K$, there exists a unique polynomial $Q(T)\in  K[T;\sigma]$  such that 
	$$P(T)=Q(T)(T-a)+P(a).$$			
	\end{lemma}
	A simple consequence of this lemma is the following: 
	$$P(a)=0\iff P(T)\in  K[T;\sigma](T-a).$$
	A polynomial $P(T)\in K[T;\sigma]$ is said to vanish on a subset $S$ of $K$ if $P(a)=0$ for all $a\in S$. In particular, if $P(a)=0$, where $a\in K$, the element $a$ is called a (right skew) root of the polynomial $P(T)\in K[T;\sigma]$. 
	
	The following formula is known as the product formula:	 
	\begin{lemma}\label{(L)ProductPQ(a)}
		For all $P(T), Q(T)\in K[T;\sigma]$, we have
		$$(PQ)(a)=\begin{cases}
			0& \text{ if } Q(a)=0,\\
			P\left( \sigma(Q(a))\, a \, Q(a)^{-1}\right) Q(a)  & \text{ if } Q(a)\neq 0,\\
		\end{cases}
		$$
		where  $(PQ)(T)=P(T)Q(T)$, as elements of $K[T;\sigma]$.  
	\end{lemma}

	One consequence of these lemmas is that a skew polynomial of degree $n$ has roots in at most $n$ different $\s$-conjugacy classes. 
	
	A subset $S$ of $K$ is called $\s$-algebraic (or simply algebraic when there is no risk of confusion) if there exists a nonzero $P(T)\in K[T;\s]$  which vanishes on $S$. One can show that every finite   subset  of $K$ is $\s$-algebraic. Moreover, 
	for any $\s$-algebraic subset $S$ of $K$, there exists a unique monic polynomial $P_S(T)\in K[T;\s]$ which vanishes on $S$, and any skew polynomial, vanishing on $S$, belongs to the left ideal $K[T;\s]P_S(T)$. The polynomial $P_S(T)$ is called the minimal (skew) polynomial of $S$ and its  degree, denoted by $rk(S)$, is called the rank of $S$. If $S$ is not $\s$-algebraic, its rank is defined to be $\infty$.  An element $a\in K$ is called P-dependent on a  $\s$-algebraic subset $S$ of $K$ if $P_S(a)=0$.  A subset $S$ of $K$ is called  P-independent  if no $a\in S$ is P-dependent on $S\setminus \{a\}$. 	
	
	Given $a_1,...,a_n\in K$, the $\s$-Vandermonde matrices  $V^\s(a_1,...,a_n)\in K^{n\times n}$ and  $V^\s_{\infty}(a_1,...,a_n)\in K^{\infty\times n}$ are defined as follows
	$$
	V^\s(a_1,...,a_n)=\begin{pmatrix}
		1&1&\cdots&1\\
		N_{1}(a_1)&N_{1}(a_2)&\cdots&N_{1}(a_n)\\
		N_{2}(a_1)&N_{2}(a_2)&\cdots&N_{2}(a_n)\\
		\vdots&\vdots&\cdots&\vdots\\
		N_{n-1}(a_1)&N_{n-1}(a_2)&\cdots&N_{n-1}(a_n)\\
	\end{pmatrix},
	$$
	$$
	V^\s_{\infty}(a_1,...,a_n)=\begin{pmatrix}
		1&1&\cdots&1\\
		N_{1}(a_1)&N_{1}(a_2)&\cdots&N_{1}(a_n)\\
		N_{2}(a_1)&N_{2}(a_2)&\cdots&N_{2}(a_n)\\
		N_{3}(a_1)&N_{3}(a_2)&\cdots&N_{3}(a_n)\\
		\vdots&\vdots&\cdots&\vdots\\
	\end{pmatrix}
	$$
	The proof of the following result can be found in \cite{lam1985general}. 
	\begin{theorem}\label{(T)Vandermonde}
		Let $a_1,...,a_n\in K$ be arbitrary.\\
		(a) The $\s$-Vandermonde matrices $V^\s(a_1,...,a_n)$ and  $V^\s_{\infty}(a_1,...,a_n)$ have equal rank, and their rank is equal to the rank of the subset $\{a_1,...,a_n\}$ of $K$. In particular, $\{a_1,...,a_n\}$ is P-independent iff $V^\s(a_1,...,a_n)$ is invertible. \\
		(b) If $\sum_{i=1}^{n}b_iN_j(a_{i})=0$ for $j=0,1,...,n-1$, then $\sum_{i=1}^{n}b_iN_j(a_{i})=0$ for all $j\geq 0$. 
	\end{theorem}
	We remark that the conclusion of Part (b) holds for any skew polynomial, that is,   $\sum_{i=1}^{n}b_iP(a_{i})=0$ for all $P(T)\in K[T;\s]$. 
	
	Theorem \ref{(T)Vandermonde} can be used  to analyze the concept of P-dependence in terms of linear dependence:  Consider the $K$-vector space $K^\infty$ of "infinite" row vectors over $K$, that is,
	$$K^\infty=\{(a_0,a_1,a_2,...)\,|\, a_0,a_1,...\in K\}.$$
	Let $N:K^*\to K^\infty$ be the map
	$$N(a)=(N_0(a),N_1(a),N_2(a),...).$$
	Using this map, we can reformulate Theorem \ref{(T)Vandermonde}:
	\begin{proposition} \label{(P)reformulationofVandermonde}
		1) A set $S\subset  K^*$ is $\s$-algebraic over $K$ iff the vector subspace 
		$N_S$ of $K^\infty$ which is generated by all $N(a)$, where $a\in S$, is finite-dimensional. Moreover, we have
		$rk(S)=\dim N_S$. \\
		2) Elements $a_1,...,a_n\in K^*$ are P-independent iff the vectors $N(a_1),...,N(a_n)$ are linearly independent over $K$. In particular, an element $a\in K^*$ is P-dependent on a subset $S$ of $K^*$ iff the row vector $N(a)$ belongs to the vector space $N_S$. 
	\end{proposition}
	The vector space $K^\infty$ can be considered as a (commutative) $K$-algebra using the multiplication
	$$(a_0,a_1,a_2,...)(b_0,b_1,b_2,...)=(a_0b_0,a_1b_1,a_2b_2,...).$$
	Since each $N_i:K\to K$ satisfies the identity $$N_i(ab)=N_i(a)N_i(b), \text{ where } a,b\in K,$$ we see that $N_{K^*}$ is a subalgebra of the $K$-algebra $K^\infty$.  More generally, $N_{S}$ is a subalgebra of the $K$-algebra $K^\infty$ if $S\subset K^*$ is a multiplicative subset of $K$, that is, $1\in S$, and $ab\in S$ when $a,b\in S$.  
	
	\end{subsection}
%%%%%%%%%%%%%%%%%%%%%%%%%%%%%%%%%%%%%%%%%%%%%%%%%%%%%%%%%%%%%%%%%   
%%%%%%%%%%%%%%%%%%%%%%%%%%%%%%%%%%%%%%%%%%%%%%%%%%%%%%%%%%%%%%%%% 
%%%%%%%%%%%%%%%%%%%%%%%%%%%%%%%%%%%%%%%%%%%%%%%%%%%%%%%%%%%%%%%%% 
%%%%%%%%%%%%%%%%%%%%%%%%%%%%%%%%%%%%%%%%%%%%%%%%%%%%%%%%%%%%%%%%% 
\begin{subsection}{Left and right roots}	
	In this part, we assume that $\s:K\to K$ is a  field automorphism. As mentioned at the beginning of this section, elements of $K[T;\s]$ can also be considered as right polynomials. As a consequence, one has the "left" analogues of the concepts and results of the preceding subsection.  More precisely, any polynomial $P(T)=\sum_{i=0}^{n}a_iT^i\in K[T;\s]$ gives rise to a left evaluation map $P_l:K\to K$ defined by
	$$P_l(a)=\sum_{i=0}^{n}\s^{-i}(a_i)N_{-i}(a),$$
	where $N_{-i}(a)=\s^{-i+1}(a)\cdots \s^{-1}(a)a$ for $a\in K$ and $i\geq 1$. 
	Moreover, the "left" version of Lemma \ref{(L)ZerosofL} holds. More precisely,  for every polynomial $P(T)\in K[T;\s]$
	and $a\in K$, there exists a unique polynomial $Q(T)\in K[T;\s]$ such that
	$$P(T)=(T-a)Q(T)+P_l(a).$$
	In particular, $P(T)\in (T-a)K[T;\s]$ iff $P_l(a)=0$ in which case $a$ is called a left (skew) root of $P(T)$. We leave it to the reader to define and prove the left analogues of the concepts and results of the preceding subsection.  We shall use the  prefixes "left" and "right" to distinguish between the concepts defined using $P_l$  and the ones presented in the preceding subsection. However, we usually drop the prefix "right" when there is no risk for confusion. The following proposition gives a relationship between the concepts of left P-dependence and right P-dependence.  

	\begin{proposition}\label{(P)FormulationofVandermonde}
		For every $a_1,...,a_n\in K^*$, the following are equivalent:\\
		(1) The right rank of the set $\{a_1,...,a_n\}$  is less than $n$, or equivalently, the elements $a_1,...,a_n$ are right P-dependent.\\
		(2) There exist $b_1,...,b_n\in K$, not all equal to zero, such that  $$b_1(1-a_1T)^{-1}+\cdots+b_n(1-a_nT)^{-1}=0.$$ 
		(3) There exist $b_1,...,b_n\in K$, not all equal to zero, such that  $$b_1(1-a_1T)^{-1}+\cdots+b_n(1-a_nT)^{-1}\in K[T;\s].$$ 
		(4) 	The left rank of the set $\{\s^{-1}(a_1^{-1}),...,\s^{-1}(a_n^{-1})\}$ is less than $n$, or equivalently, the elements $\s^{-1}(a_1^{-1}),...,\s^{-1}(a_n^{-1})$ are left P-dependent. 	
		
	\end{proposition}
	\begin{proof}
		The equivalence of (1) and (2) follows from Theorem \ref{(T)Vandermonde} and the fact that $(1-aT)^{-1}$, as an element of $K[[T;\s]]$, is the series
		$$(1-aT)^{-1}=\sum_{i=0}^\infty N_i(a)T^i.$$
		Clearly, (2) implies (3). To prove (3)$\implies$(4), let there be $b_1,...,b_n\in K$, not all equal to zero, such that
		$$b_1(1-a_1T)^{-1}+\cdots+b_n(1-a_nT)^{-1}=P(T)\in K[T;\s].$$
		Without loss of generality, we may assume that $b_n\neq 0$. This equality can be rewritten as follows
			$$b_1\s^{-1}(a_1^{-1})(\s^{-1}(a_1^{-1})-T)^{-1}+\cdots+b_n\s^{-1}(a_n^{-1})(\s^{-1}(a_n^{-1})-T)^{-1}=P(T).$$
		Let $Q(T)\in K[T;\s]$ have $\s^{-1}(a_1^{-1}),...,\s^{-1}(a_n^{-1})$ as left zeros, that is, $$Q_l(\s^{-1}(a_i^{-1}))=0,$$ for $1\leq i<n$. It follows that 
	\begin{equation*}
		\begin{aligned}
			b_n\s^{-1}(a_n^{-1})(\s^{-1}(a_n^{-1})-T)^{-1}Q(T) ={} & P(T)Q(T)\\
			+& \sum_{i=1}^{n-1}b_i\s^{-1}(a_i^{-1})(T-\s^{-1}(a_i^{-1}))^{-1}Q(T)\\
		\end{aligned}
	\end{equation*}
	belongs to $K[T;\s]$. Therefore, we see that
		 $Q_l(\s^{-1}(a_n^{-1}))=0$. Therefore, we conclude that $\s^{-1}(a_n^{-1})$ is left P-dependent on the set $\{\s^{-1}(a_1^{-1}),...,\s^{-1}(a_{n-1}^{-1})\}$. In particular, the left rank of the set $\{\s^{-1}(a_1^{-1}),...,\s^{-1}(a_n^{-1})\}$ is less than $n$. This completes the proof of  (3)$\implies$(4). It remains to prove (4)$\implies$(2). Let $\s^{-1}(a_1^{-1}),...,\s^{-1}(a_n^{-1})$ be left P-dependent. 	Without loss of generality, we may assume that $\s^{-1}(a_n^{-1})$ is P-dependent on the set $\{\s^{-1}(a_1^{-1}),...,\s^{-1}(a_{n-1}^{-1})\}$. 	Let $Q(T)\in K[T;\s]$ be the minimal left polynomial of $\{\s^{-1}(a_1^{-1}),...,\s^{-1}(a_{n-1}^{-1})\}$. Note that the degree of $Q(T)$ is less than $n$. This fact, together with the fact that 
		 $$(\s^{-1}(a_n^{-1})-T)^{-1}Q(T)\in K[T;\s],$$
		 for all $1\leq i\leq n$, imply that there are  $b_1,...,b_n\in K$, not all equal to zero, such that  $$b_1(\s^{-1}(a_1^{-1})-T)^{-1}Q(T)+\cdots+b_n(\s^{-1}(a_n^{-1})-T)^{-1}Q(T)=0.$$
		 Since $K(T;\s)$ is a skew field, it follows that 
		 $$b_1(\s^{-1}(a_1^{-1})-T)^{-1}+\cdots+b_n(\s^{-1}(a_n^{-1})-T)^{-1}=0,$$
		 or equivalently, 
		 $$b_1\s^{-1}(a_1)(1-a_1T)^{-1}+\cdots+b_n\s^{-1}(a_n)(1-a_nT)^{-1}=0,$$
		 and we are done. 
	\end{proof}
	 As a consequence of this proposition, we have
	 \begin{corollary}
	 	A set $S\subset K^*$ is right $\s$-algebraic iff the set 
	 	$$\widehat{S}=\{\s^{-1}(a^{-1})\, |\, a\in S\},$$
	 	is left $\s$-algebraic. Furthermore, the right rank of an arbitrary subset $S$ of $K^*$  is equal to the left rank of $\widehat{S}$. 
	 \end{corollary}
	In the case when $\s$ is an involution, that is, $\s^2=1_K$, we have
	$$\widehat{S}=\{\s(a^{-1})\, |\, a\in S\}.$$
	As an example, a set $S\subset \mathbb{C}^*$ is right $\s$-algebraic iff the set 
	$$\widehat{S}=\left\lbrace \frac{1}{\bar{c}}\, |\, c\in S\right\rbrace ,$$
	is left $\s$-algebraic. Here, $\s:\mathbb{C}\to \mathbb{C}$ is the complex conjugation map. 

	Although Proposition \ref{(P)FormulationofVandermonde} gives a relationship between left and right ranks, it does not provide any relationship between minimal left and right  polynomials.  In order to derive such a relationship, we need the following result which deals with anti-automorphisms of rings of skew polynomials. Recall that a map $\alpha:R\to R$, where $R$ is a ring,  is called an anti-automorphism, if it is an additive group isomorphism, $\alpha(1)=1$, and $\alpha(ab)=\alpha(b)\alpha(a)$ for all $a,b\in R$. In the following proposition, we give a characterization of anti-automorphisms of $K[T;\s]$. For the case of automorphisms of skew polynomial rings over (not necessarily commutative) rings, see \cite{chuang2013automorphisms}.  
	\begin{proposition}\label{(P)Anti-automorphism}
		Suppose that $\s:K\to K$ is an automorphism not equal to the identity map $1_K$. Let $\alpha:K[T;\s]\to K[T;\s]$ be a map. If $\alpha$ is an anti-automorphism, then the following hold:\\
		(1) $\alpha(K)= K$, and the restriction $\alpha_0$ of $\alpha$ to $K$ is an automorphism of $K$. \\
		(2) There exists an element $a_0\in K^*$ such that $\alpha(T)=a_0 T$.\\
		(3) $\s\alpha_0\s=\alpha_0$. \\
		Conversely, if $\alpha_0:K\to K$ is an automorphism of $K$ which satisfies Property (3), then, for every $a_0\in K^*$, there exists a unique anti-automorphism \newline $\alpha:K[T;\s]\to K[T;\s]$, with $\alpha(T)=a_0T$, whose restriction to $K$ is $\alpha_0$.  
		
	\end{proposition} 
	\begin{proof}
		Let $\alpha:K[T;\alpha]\to K[T;\alpha]$ be an anti-automorphism. First, we prove (1). Since the set of units of $K[T;\s]$ is $K^*$, we see that $\alpha(K)\subset K$. Similarly, we have $\alpha^{-1}(K)\subset K$ since $\alpha^{-1}$ is an anti-automorphism of $K[T;\alpha]$. It follows that $\alpha(K)=K$. Clearly, $\alpha_0:K\to K$ is an automorphism of $K$ since $K$ is commutative. This completes the proof of (1).  A simple argument using the degree function reveals that $\alpha(T)=a_0T+b_0$ for some $a_0,b_0\in K$ with $a_0\neq 0$.
		For $a\in K$, we have
		$$Ta=\s(a)T\implies \alpha(Ta)=\alpha (\s(a)T) \implies
		\alpha_0(a)\alpha(T)=\alpha (T)\alpha_0 (\s(a)),$$
		which in turn implies that 
		$$ \alpha_0(a)a_0T+\alpha_0(a)b_0=a_0T\alpha_0 (\s(a))+b_0\alpha_0(\s(a))=a_0\s(\alpha_0 (\s(a)))T+b_0\alpha_0(\s(a)).$$
		It follows that
		$$\forall a\in K, \, \alpha_0(a)a_0=a_0\s(\alpha_0 (\s(a))) \text{ and } \alpha_0(a)b_0= b_0\alpha_0 (\s(a)).$$
		The first relation proves (3)  because $a_0\neq 0$. If $b_0\neq 0$, then $\alpha_0$ must be equal to $\alpha_0\s$ which implies that $\s=1_K$ since $\alpha$ is a bijection. This contradicts the assumption that $\s\neq 1_K$. This completes the proof of (2). 
		
		To prove the second statement, let $\alpha_0:K\to K$ be an automorphism satisfying (1) and (3). Given $a_0\in K^*$, we define $\alpha:K[T;\alpha]\to K[T;\alpha]$, as follows
		$$\alpha(\sum_{i=0}^nb_iT^i)=\sum_{i=0}^n\s^i(\alpha_0(b_i))N_i(a_0)T^i.$$
		It is easy to see that $\alpha$ is an additive group isomorphism. Since $\alpha(1)=1$, $\alpha(T)=a_0T$, and the restriction of $\alpha$ to $K$ is $\alpha_0$, we only need to prove that 
		$$\alpha(P(T)Q(T))=\alpha(Q(T))\alpha(P(T)),$$
		for all $P(T),Q(T)\in K[T;\s]$. By the additivity of $\alpha$, it is enough to show that
		$$\alpha \left( (aT^m)(bT^n)\right) =\alpha \left( bT^n\right) \alpha \left( aT^m\right),$$
		for all $a,b\in K$ and $m,n\geq 0$. We have
		$$\alpha \left( (aT^m)(bT^n)\right)=\alpha \left( a\s^m(b)T^{m+n}\right) =\s^{m+n}\left( \alpha_0\left(a\s^m(b) \right) \right)N_{m+n}(a_0) T^{m+n},$$
		and
		\begin{align*}
			\alpha \left( bT^n\right) \alpha \left( aT^m\right) = & \, \s^{n}\left( \alpha_0\left(b \right) \right)N_n(a_0) T^{n}\s^{m}\left( \alpha_0\left(a\right) \right) N_m(a_0)T^{m}\\
			=&\, \s^{n}\left( \alpha_0\left(b \right) \right)N_n(a_0) s^{n}\left( \s^{m}\left( \alpha_0\left(a\right) \right) N_m(a_0)\right) T^{m+n}.
		\end{align*}
		Therefore, we need to show that
		$$
		\s^{m+n}\left( \alpha_0\left(a\s^m(b) \right) \right)N_{m+n}(a_0) =
		\s^{n}\left( \alpha_0\left(b \right) \right)N_n(a_0) \s^{n}\left( \s^{m}\left( \alpha_0\left(a\right) \right) N_m(a_0)\right),
		$$
		for all $a,b\in K$ and $m,n\geq 0$. Using Equation \ref{(Id)Ni+j} and the fact that $\s$ and $\alpha_0$ are ring homomorphisms, we can simplify this equation to obtain
		$$
		\s^{m+n}\left( \alpha_0\left(\s^m(b) \right) \right) =
		\s^{n}\left( \alpha_0\left(b \right) \right).
		$$
		It is now easy to use $\s\alpha_0 \s =\alpha_0$ to show that this relation holds for all $b\in K$ and $m,n\geq 0$. The uniqueness part follows from the fact that $K$ and $T$ generate $K[T;\s]$ as a ring. This completes the proof of the second statement. 
	\end{proof}
	Let $AAut(K,\s)$ denote the set of all anti-automorphisms of $K[T;\s]$ and $Aut_\s(K)$ denote the set of all automorphism $\alpha_0:K\to K$ such that $\s\alpha_0\s=\alpha_0$. Using Proposition \ref{(P)Anti-automorphism}, we obtain  a 1-1 correspondence between $AAut(K,\s)$ and the set $Aut_\s(K)\times K^*$, given by $\alpha\mapsto (\alpha|_{K},\alpha(T)T^{-1})$. It may happen that $Aut_\s(K)$ is empty, in which case $AAut(K,\s)$ is empty as well. As an example, the reader can verify that if $K$ is any cyclic extension of the field of rational numbers, and $\s$ is a generator of the Galois group of $K$ over $\mathbb{Q}$, then $AAut(K,\s)$  is nonempty iff $\s^2=1_K$ (in the case $\s^2=1_K$, we have $1_K\in Aut_\s(K)$). If $AAut(K,\s)$ is nonempty, then we have the following result concerning left and right roots:
	\begin{proposition}
		Let $\alpha\in AAut(K,\s)$ satisfy $\alpha(T)=T$, and set $\alpha_0=\alpha|_K$. Then, the following hold:\\
		(1) $a\in K$ is a left root of  a polynomial $P(T)\in K[T;\s]$ iff $\alpha_0(a)$ is a right root of $\alpha(P(T))$. \\
		(2) A polynomial $P(T)\in K[T;\s]$ is the  minimal left polynomial of a subset of $K$ iff $\alpha(P(T))$ is the  minimal right polynomial of a subset of $K$. \\
		(3) The assignment $S\mapsto \alpha(S)$ establishes a 1-1 correspondence between the set of all left $\s$-algebraic subsets of $K$ and the set of all right $\s$-algebraic subsets of $K$. Moreover, the left rank of $S$ is equal to the right rank of $\alpha(S)$.
	\end{proposition}
	 \begin{proof}
	 	(1) If $a$ is a left root of $P(T)\in K[T;\s]$, then 
	 	$$P(T)=(T-a) Q(T),$$
	 	for some $Q(T)\in K[T;\s]$. Applying $\alpha$ to this relation gives
	 	$$\alpha(P(T))=\alpha((T-a)Q(T))=\alpha(Q(T))\alpha(T-a)=\alpha(Q(T))(T-\alpha_0(a)),$$
	 	showing that $\alpha(a_0)$ is a right root of $\alpha(P(T))$. The converse is proved in a similar fashion. Now, Parts
	 	(2) and (3) follow easily from (1).
	 \end{proof}
 	In the case when $\s$ is an involution, the identity map of $K[T;\s]$ belongs to $AAut(K,\s)$. In this case, an application of this proposition gives the following: (1) 
 	$a$ is a left root of $\sum_{i=0}^n a_iT^i$ iff $a$ is a right root  of $\sum_{i=0}^nT^ia_i$; (2) $S\subset K$ is left $\s$-algebraic iff it is right $\s$-algebraic, and moreover, the left rank of $S$ is equal to its right rank. 
\end{subsection}
%%%%%%%%%%%%%%%%%%%%%%%%%%%%%%%%%%%%%%%%%%%%%%%%%%%%%%%%%%%%%%%%%   
%%%%%%%%%%%%%%%%%%%%%%%%%%%%%%%%%%%%%%%%%%%%%%%%%%%%%%%%%%%%%%%%% 
%%%%%%%%%%%%%%%%%%%%%%%%%%%%%%%%%%%%%%%%%%%%%%%%%%%%%%%%%%%%%%%%% 
\begin{subsection}{Values of skew polynomials}
	In this subsection, we present some elementary results concerning values of skew polynomials. 	As opposed to ordinary polynomials in an indeterminate over commutative fields, nonzero skew polynomials over commutative fields may have infinitely many roots. For example, the polynomial 
	$$T^2-1\in \mathbb{C}[T;\,\bar\, \,],$$
	vanishes on the unit circle $\{z\in\mathbb{C}\,|\,|z|=1\}$.  However, we have the following result, which is well-known for ordinary polynomials over commutative fields. The result also holds for  polynomials in a central indeterminate over a skew field  (see \cite[Theorem 16.7]{lam1991first}). 
	\begin{proposition}
		Let $(K,\sigma)$ be an infinite $\s$-field and  $P(T)\in K[T;\sigma]$. If $P(a)=0$ for all $a\in K$, then $P(T)=0$. 
	\end{proposition}
	\begin{proof}
		Let $F$ be the fixed field of $\s$. Suppose, on the contrary, that a skew polynomial $P(T)=\sum_{i=0}^n a_i T^i\neq 0$ vanishes on $K$.  Then, all the $\s$-conjugacy classes of $K$ must be $\s$-algebraic. It follow from   \cite[Theorem 5.17]{lamleroy1988algebraic} that the degree $[K:F]$ of the extension $K/F$ is finite. Therefore, $F$ must be an infinite field. Since
		$P(c)=\sum_{i=0}^n a_i c^i,$
		where $c\in F$, the nonzero ordinary polynomial $\sum_{i=0}^n a_i x^i\in K[x],$ must have infinitely many zeros over the commutative field $K$, a contradiction.   	   
	\end{proof}
	This theorem does not hold for finite $\s$-fields:  Since any finite subset of a $\s$-field is $\s$-algebraic, we see that there is a nonzero polynomial $P(T)\in K[T;\sigma]$ vanishing on $K$ if $(K,\sigma)$ is a finite $\s$-field, in which case, the nonconstant polynomial $P(T)+1$ has no roots in $K$. We note in passing that if $(K,\s)$ is a finite $\s$-field, then the set of all $P(T)\in K[T;\sigma]$  which vanish on $K$ is a nonzero (two-sided) ideal of $ K[T;\sigma]$. 

	Next, we prove a version of the Bray--Whaples theorem for skew polynomials, see \cite{braywhaples1983polynomials} for the result which deals with polynomials in a central indeterminate over a skew field. 
	\begin{proposition}\label{(P)Bray_Whaples}
		Let $a_1,...,a_n\in K$ be mutually nonconjugate over $K$. Then, there is a unique monic skew polynomial $P(T)\in K[T;\s]$ of degree $n$ whose roots are exactly $a_1,...,a_n$. Moreover,  if $$Q(a_1)=\cdots=Q(a_n)=0,$$ where  $Q(T)\in K[T;\s]$, then $Q(T)\in  K[T;\s]P(T)$.   
	\end{proposition}
	\begin{proof}
		Let $P(T)\in K[T;\s]$ be the minimal skew polynomial of $\{a_1,...,a_n\}$. Since $a_1,...,a_n$ are mutually nonconjugate, the degree of $P(T)$ must be $n$. Therefore, it remains to show that $P(T)$ has no  roots other than $a_1,...,a_n$. This holds trivially if $n=1$. So, let $n>1$. Suppose, on the contrary, that there exists $b\in K$ not equal to $a_1,...,a_n$ such that $P(b)=0$. By Lemma \ref{(L)DivisionP(a)}, we have
		$$P(T)=Q(T)(T-b),$$
		for some $Q(T)\in K[T;\s]$.  Since $b\neq a_i$ for all $i$, 
		it follows from Lemma \ref{(L)ProductPQ(a)} that  the skew polynomial $Q(T)$ has the following roots $$c_i=\sigma(a_i-b) a_i(a_i-b)^{-1} \text{ where } i=1,...,n.$$  Clearly, each $c_i$ is conjugate to $a_i$. Therefore, the roots $c_1,...,c_n$  of $Q(T)$ are mutually nonconjugate, contradicting the fact that a skew polynomial of degree less than $n$ cannot have roots in $n$ different conjugacy classes.  The second statement follows from the fact that $P(T)$ is the minimal  polynomial of $\{a_1,...,a_n\}$.  
	\end{proof}
	In general, the classical Lagrange interpolation does not work for skew polynomials since an element $a\in K$ can be P-dependent on a subset not containing $a$ (see the discussion after Proposition \ref{(P)Interpolation}). However, we have the following version of interpolation for skew polynomials, which can be regarded as a generalization of Proposition \ref{(P)Bray_Whaples}. The problem of interpolation for polynomials in a central indeterminate over a skew field is treated in \cite{bolotnikov2020lagrange}.
	 \begin{proposition}\label{(P)Interpolation}
	 	Let $S=\{a_1,...,a_n\}\subset K$ be P-independent. For any $b_1,...,b_n\in K$,  there  exists a unique skew polynomial $P(T)\in K[T;\s]$ of degree less than $n=rk (S)$ such that $P(a_i)=b_i$ for all $i$. Moreover, a  polynomial $Q(T)\in K[T;\s]$  satisfies $Q(a_i)=b_i$ for all $i$, iff $$Q(T)\in P(T)+K[T;\s]P_S(T).$$
	 \end{proposition}
	 \begin{proof}
	 	We note that $rk(S)=n$ by  \cite[Lemma 12]{lam1985general}. Let $P_i(T)\in K[T;\s]$ be the minimal skew polynomial of $S\setminus \{a_i\}$. The degree of each $P_i(T)$ is $rk (S)-1=n-1$. Since $a_1,...,a_n$ are P-independent, $P_i(a_i)$ cannot be zero. It is easy to verify that the polynomial 
	 	$$P(T)=\sum_{i=1}^nb_iP_i(a_i)^{-1}P_i(T),$$
	 	satisfies the desired property. The uniqueness of $P(T)$ follows from the fact that  the degree of each $P_i(T)$ is less than $rk (S)$. 
	 \end{proof}  
 	A consequence of this proposition is that if $a\in K$ is P-dependent on a P-independent set $S=\{a_1,...,a_n\}$, then for any $Q(T)\in K[T;\s]$, the value of $Q(a)$ is completely determined by the values $Q(a_1),...,Q(a_n)$. In fact, we have 
 	$$Q(a)=\sum_{i=1}^nQ(a_i)P_i(a_i)^{-1}P_i(a),$$
	where $P_i(T)$ is the minimal skew polynomial of $S\setminus \{a_i\}$. More generally, one can easily show that the value of $Q(a)$ is completely determined by the values $Q(a), a\in S$, if $a$ is P-dependent on a (not necessarily P-independent) $\s$-algebraic set $S\subset K$.   	 
\end{subsection}
%%%%%%%%%%%%%%%%%%%%%%%%%%%%%%%%%%%%%%%%%%%%%%%%%%%%%%%%%%%%%%%%%   
%%%%%%%%%%%%%%%%%%%%%%%%%%%%%%%%%%%%%%%%%%%%%%%%%%%%%%%%%%%%%%%%% 
%%%%%%%%%%%%%%%%%%%%%%%%%%%%%%%%%%%%%%%%%%%%%%%%%%%%%%%%%%%%%%%%% 
\begin{subsection}{Vieta's formulas for skew polynomials }
	The classical formulas of Vieta express the coefficients of an ordinary polynomial over a commutative field in terms of its roots. 
	 Here, we present analogous formulas for skew polynomials.    	
	The formulas in the second part of the following result can be regarded as analogues of Vieta's formulas (see  \cite{delenclos2007noncommutative}, \cite[Section 7.2]{Gelfandetal1995study} and \cite{gelfand1996noncommutative} for a more general discussion). 
	\begin{theorem}\label{(T)Vieta}
		Let $S=\{a_1,...,a_n\}$ be P-independent with minimal polynomial $P_S(T)=\sum_{i=0}^nb_i T^n\in K[T;\s]$. \\
		(a) For every $a\in K$, we have
		$$P_S(a)=\frac{ \det \left( V^\s(a_1,...,a_n,a)\right)  }{\det \left( V^\s(a_1,...,a_n)\right) }.$$
		(b) For $i=0,...,n$, 
		\begin{equation}\label{(F)MinimalVandermonde}
			b_i=\frac{(-1)^{n+i}}{\det \left( V^\s(a_1,...,a_n)\right)}\begin{vmatrix}
				1&1&\cdots&1\\
				N_{1}(a_1)&N_{1}(a_2)&\cdots&N_{1}(a_n)\\
				\vdots&\vdots&\cdots&\vdots\\
				N_{i-1}(a_1)&N_{i-1}(a_2)&\cdots&N_{i-1}(a_n)\\
				N_{i+1}(a_1)&N_{i+1}(a_2)&\cdots&N_{i+1}(a_n)\\
				\vdots&\vdots&\cdots&\vdots\\
				N_{n}(a_1)&N_{n}(a_2)&\cdots&N_{n}(a_n)\\
			\end{vmatrix}
		\end{equation}
	Moreover, 
	$b_0=(-1)^{n}\, \frac{\s\left( \det \left( V^\s(a_1,...,a_n)\right)\right) }{\det \left( V^\s(a_1,...,a_n)\right)}\, a_1a_2\cdots a_n.$
%	$$b_{n-1}=-\, \frac{\s\left( \det \left( V^\s(a_1,...,a_n)\right)\right) }{\det \left( V^\s(a_1,...,a_n)\right)}\, (a_1+a_2+\cdots +a_n).$$

	\end{theorem}
	 \begin{proof}
	 	Using cofactor expansion along the last column of the matrix $V^\s(a_1,...,a_n,a)$, we obtain
	 	$$\det \left( V^\s(a_1,...,a_n,a)\right)=\sum_{i=0}^n c_i N_i(a),$$
	 	where $$c_i=(-1)^{n+i}\begin{vmatrix}
	 		1&1&\cdots&1\\
	 		N_{1}(a_1)&N_{1}(a_2)&\cdots&N_{1}(a_n)\\
	 		\vdots&\vdots&\cdots&\vdots\\
	 		N_{i-1}(a_1)&N_{i-1}(a_2)&\cdots&N_{i-1}(a_n)\\
	 		N_{i+1}(a_1)&N_{i+1}(a_2)&\cdots&N_{i+1}(a_n)\\
	 		\vdots&\vdots&\cdots&\vdots\\
	 		N_{n}(a_1)&N_{n}(a_2)&\cdots&N_{n}(a_n)\\
	 	\end{vmatrix}.$$  Seting $Q(T)=\sum_{i=0}^nc_i T^n$, we have
	 	$$Q(a)=\sum_{i=0}^nc_i N_i(a)=\det \left( V^\s(a_1,...,a_n,a)\right),$$
	 	for all $a\in K$. By Part (a) of Theorem \ref{(T)Vandermonde}, $Q(a_i)=0$ for all $i=1,...,n$. It follows that   
	 	$$Q(T)\in K[T;\s]P_S(T).$$
	 	Since $P_S(T)$ and $Q(T)$ have the same degree, and $c_n=\det (V^\s(a_1,...,a_n))$, we conclude that
	 	$$Q(T)=V^\s(a_1,...,a_n)P_S(T),$$
	 and
	 $$P_S(a)=\frac{ \det \left( V^\s(a_1,...,a_n,a)\right)  }{\det \left( V^\s(a_1,...,a_n)\right) },$$
	 for all $a\in K$.\
	 
	 (b) Formula \ref{(F)MinimalVandermonde} follows from (a). To prove the last formula, we set $i=0$ in Formula \ref{(F)MinimalVandermonde} and use Identity \ref{(Id)Ni+j}. We see that $(-1)^{n}b_0\det \left( V^\s(a_1,...,a_n)\right)$ is equal to
	 \begin{align*}
	 	  & \begin{vmatrix}
	 		N_{1}(a_1)&N_{1}(a_2)&\cdots&N_{1}(a_n)\\
	 		N_{2}(a_1)&N_{2}(a_2)&\cdots&N_{2}(a_n)\\
	 		\vdots&\vdots&\cdots&\vdots\\
	 		N_{n}(a_1)&N_{n}(a_2)&\cdots&N_{n}(a_n)\\
	 	\end{vmatrix}\\
 	=& \begin{vmatrix}
 		a_1&a_2&\cdots&a_n\\
 		a_1\s(N_{1}(a_1))&a_2\s(N_{1}(a_2))&\cdots&a_n\s(N_{1}(a_n))\\
 		a_1\s(N_{2}(a_1))&a_2\s(N_{2}(a_2))&\cdots&a_n\s(N_{2}(a_n))\\
 		\vdots&\vdots&\cdots&\vdots\\
 		a_1\s(N_{n-1}(a_1))&a_2\s(N_{n-1}(a_2))&\cdots&a_n\s(N_{n-1}(a_n))\\
 	\end{vmatrix}\\
 =& \begin{vmatrix}
 	1&1&\cdots&1\\
 	\s(N_{1}(a_1))&\s(N_{1}(a_2))&\cdots&\s(N_{1}(a_n))\\
 	\s(N_{2}(a_1))&a_2\s(N_{2}(a_2))&\cdots&\s(N_{2}(a_n))\\
 	\vdots&\vdots&\cdots&\vdots\\
 	\s(N_{n-1}(a_1))&\s(N_{n-1}(a_2))&\cdots&\s(N_{n-1}(a_n))\\
 \end{vmatrix}a_1a_2\cdots a_n\\
=& \s \left( \begin{vmatrix}
	1&1&\cdots&1\\
	N_{1}(a_1)&N_{1}(a_2)&\cdots&N_{1}(a_n)\\
	N_{2}(a_1)&N_{2}(a_2)&\cdots&N_{2}(a_n)\\
	\vdots&\vdots&\cdots&\vdots\\
	N_{n-1}(a_1)&N_{n-1}(a_2)&\cdots&N_{n-1}(a_n)\\
\end{vmatrix}\right) a_1a_2\cdots a_n\\
=& \s\left( \det \left( V^\s(a_1,...,a_n)\right)\right)  a_1a_2\cdots a_n,
	 \end{align*}
which proves the desired formula. 	 
	 \end{proof}
 We note that Part (a) of this theorem provides a method for finding the minimal polynomial of any finite P-independent set. In particular, one can apply this method to find the polynomials  $P_i(T)$'s used in the proof of Proposition \ref{(P)Interpolation}.  	
	 
\end{subsection}
%%%%%%%%%%%%%%%%%%%%%%%%%%%%%%%%%%%%%%%%%%%%%%%%%%%%%%%%%%%%%%%%% 

\end{section} 
%%%%%%%%%%%%%%%%%%%%%%%%%%%%%%%%%%%%%%%%%%%%%%%%%%%%%%%%%%%%%%%%
%%%%%%%%%%%%%%%%%%%%%%%%%%%%%%%%%%%%%%%%%%%%%%%%%%%%%%%%%%%%%%%%
%%%%%%%%%%%%%%%%%%%%%%%%%%%%%%%%%%%%%%%%%%%%%%%%%%%%%%%%%%%%%%%%
%%%%%%%%%%%%%%%%%%%%%%%%%%%%%%%%%%%%%%%%%%%%%%%%%%%%%%%%%%%%%%%%
%%%%%%%%%%%%%%%%%%%%%%%%%%%%%%%%%%%%%%%%%%%%%%%%%%%%%%%%%%%%%%%%
%%%%%%%%%%%%%%%%%%%%%%%%%%%%%%%%%%%%%%%%%%%%%%%%%%%%%%%%%%%%%%%%
%%%%%%%%%%%%%%%%%%%%%%%%%%%%%%%%%%%%%%%%%%%%%%%%%%%%%%%%%%%%%%%%
%%%%%%%%%%%%%%%%%%%%%%%%%%%%%%%%%%%%%%%%%%%%%%%%%%%%%%%%%%%%%%%%
%%%%%%%%%%%%%%%%%%%%%%%%%%%%%%%%%%%%%%%%%%%%%%%%%%%%%%%%%%%%%%%%
%%%%%%%%%%%%%%%%%%%%%%%%%%%%%%%%%%%%%%%%%%%%%%%%%%%%%%%%%%%%%%%%
%%%%%%%%%%%%%%%%%%%%%%%%%%%%%%%%%%%%%%%%%%%%%%%%%%%%%%%%%%%%%%%%
%%%%%%%%%%%%%%%%%%%%%%%%%%%%%%%%%%%%%%%%%%%%%%%%%%%%%%%%%%%%%%%%
%%%%%%%%%%%%%%%%%%%%%%%%%%%%%%%%%%%%%%%%%%%%%%%%%%%%%%%%%%%%%%%%
%%%%%%%%%%%%%%%%%%%%%%%%%%%%%%%%%%%%%%%%%%%%%%%%%%%%%%%%%%%%%%%%
%%%%%%%%%%%%%%%%%%%%%%%%%%%%%%%%%%%%%%%%%%%%%%%%%%%%%%%%%%%%%%%%
%%%%%%%%%%%%%%%%%%%%%%%%%%%%%%%%%%%%%%%%%%%%%%%%%%%%%%%%%%%%%%%%
\begin{section}{$\s$-Extensions}
This section deals with the notion of $\s$-algebraicity in the context of $\s$-extensions. Since it is more convenient to use the language of difference rings, we begin by reviewing basic facts about difference rings. For a detailed account of the theory of difference rings, we refer the reader to the book \cite{levin2008difference}.  

	By a difference ring, we mean a pair $(R,\s_R)$ where  $R$ is a commutative ring  and  $\s_R:R\to R$ is an injective endomorphism. The ring $R$ is called the underlying ring of the difference ring $(R,\s)$ and the endomorphism $\s_R$ is called its structure map. A difference ring is also called a $\s$-ring, a term we mainly use in this paper. By abuse of notation, we often write $\s$ in place of $\s_R$ if there is no risk of confusion. Note that any ring equipped with its identity map gives rise to a $\s$-ring which we call a trivial $\s$-ring. In what follows, we present some definitions and facts regarding difference rings which are relevant to this work. 
	\\
	(1) A $\s$-ring  is called a $\s$-field (resp., $\s$-domain) if its underlying ring is a field (resp., an integral domain). \\
	(2) By a $\s$-ideal of a $\s$-ring $(R,\s)$, we mean an ideal $I$ of $R$ which satisfies 
	$$\forall r\in R,\, r\in I \iff \s(r)\subset I.$$   
	For any $\s$-ideal $I$ of a $\s$-ring $(R,\s)$, we can consider a $\s$-ring structure on $R/I$ defined by $r+I\mapsto \s(r)+I$, which is called the quotient of $(R,\s)$ modulo $I$, and denoted by  $(R/I,\s)$.  \\
	(3) A $\s$-homomorphism $\phi:(R,\s_R)\to (S,\s_S)$ between $\s$-rings is a ring homomorphism of $\phi:R\to S$ between the underlying rings which satisfies $\phi\s_R=\s_S\phi$. It is easy to see that $\s$-rings and $\s$-homomorphisms form a category. If $I$ is a $\s$-ideal of $(R,\s)$, then the quotient map  $q:R\to R/I$ gives rise to a $\s$-homomorphism $q:(R,\s)\to (R/I,\s)$. If $\phi:(R,\s_R)\to (S,\s_S)$ is a $\s$-homomorphism, then $\ker \phi$ is a $\s$-ideal of $(R,\s)$, and the map $R/\ker \phi\to S$ is an (injective)  $\s$-homomorphism. \\
	(4) By a $\s$-subring of a $\s$-ring $(R,\s)$, we mean a subring $S$  of $R$ which satisfies $\s(S)\subset S$, in which case, the inclusion map becomes a $\s$-homomorphism $i:(S,\s)\to (R,\s)$. In particular, if both $S$ and $R$ are fields, $(S,\s)$ is called a $\s$-subfield of $(R,\s)$, and $(R,\s)$ is called an \textit{extension $\s$-field} (or simply \textit{$\s$-extension})  of $(S,\s)$. \\
	(5) Let $(S,\s)$ be a $\s$-subring of a $\s$-ring $(R,\s)$. For a subset $X$ of  $(R,\s)$, the intersection of all $\s$-subrings of $(R,\s)$ containing $S\cup X$, is itself a $\s$-subring of $(R,\s)$ which is called the $\s$-subring of $(R,\s)$ generated by $X$ over $(S,\s)$, and denoted by $S[X]_{\s}$. In case, $X=\{x_1,...,x_n\}$ is finite, we denote  $S[X]_{\s}$ by $S[x_1,...,x_n]_\s$. \\
	(6) Let $(K,\s)$ be a $\s$-subfield of a $\s$-field $(L,\s)$. For a subset $X$ of  $(L,\s)$, the intersection of all $\s$-subfields of $(L,\s)$ containing $K\cup X$ is itself a $\s$-field of $(L,\s)$ which is called the $\s$-subfield of $(L,\s)$ generated by $X$ over $(K,\s)$, and denoted by $K\langle X \rangle$. In case, $X=\{x_1,...,x_n\}$ is finite, we denote  $K\langle X \rangle$ by $K\langle x_1,...,x_n \rangle$. \\
	(7) A $\s$-ring $(R,\s)$ is called inversive if $\s:R\to R$ is onto. It is known that for any $\s$-ring $(R,\s)$, there exists a $\s$-ring $(U,\s)$ containing $(R,\s)$ as a $\s$-subring such that any $\s$-homomorphism $\psi:(R,\s)\to (S,\s_S)$, where $(S,\s_S)$ is inversive, has a unique extension to a $\s$-homomorphism $(U,\s)\to (S,\s_S)$. The $\s$-ring $(U,\s)$, called a universal inversive closure of $(R,\s)$, is unique up to $\s$-isomorphism.  Moreover, $(U,\s)$ has the following property:
	$$\forall u\in U,\, \exists n\geq 1, \, \s^n(u)\in R.$$
	Also, $(U,\s)$ is a $\s$-field (resp., $\s$-domain) if $(R,\s)$ is  a $\s$-field (resp., $\s$-domain). \\
	(8) A multiplicative $\s$-subset of a $\s$-ring $(R,\s)$ is a multiplicative subset $S$ of $R$ which satisfies  
	$$\forall r\in R,\, \left( r\in S \implies \s(r)\subset S\right) .$$
	Given a multiplicative   $\s$-subset of a $\s$-domain $(R,\s)$, the map $$rs^{-1}\mapsto \s(r)\s(s)^{-1},$$ is a well-defined map on the localization ring $S^{-1}R$, turning $S^{-1}R$ into a $\s$-domain. Moreover, the ring homomorphism  
	$$R\to S^{-1}R,\, r\mapsto r1^{-1},$$
	is a $\s$-homomorphism. In particular, if $(R,\s)$ is a $\s$-domain, then  the quotient field $Frac(R)$ of $R$ acquires a $\s$-field structure for which the inclusion map $R\to Frac(R)$ is a $\s$-homomorphism. The resulting $\s$-field $(Frac(R),\s)$ is called the $\s$-field of fractions of $(R,\s)$ . \\
	(9) Let $(R,\s)$ be a $\s$-field. Consider the ring 
	$$R\{x\}:=R[x(0),x(1),x(2),...],$$ 
	of polynomials over $R$ in (commuting) indeterminates $x(0),x(1),...$. There is a unique endomorphism 
	$\s:R\{x\}\to R\{x\}$ which extends $\s:R\to R$ and satisfies $\s(x(i))=x(i+1)$ for $i\geq 0$. Clearly, $\s:R\{x\}\to R\{x\}$  is injective. Therefore, the pair $(R\{x\},\s\})$ defines a $\s$-ring which includes $(R,\s)$ as a $\s$-subring. The $\s$-ring  $(R\{x\},\s\})$ is called  
	the ring of difference polynomials  in the difference indeterminate $x$ over $(R,\s)$. It has the following familiar property: For any $\s$-homomorphism $\phi:(R,\s)\to (S,\s_S)$ and any element $s\in S$, there exists a unique $\s$-homomorphism $\phi:(R\{x\},\s\})\to (S,\s_S)$ which extends $\phi:(R,\s)\to (S,\s_S)$  and sends $x(0)$ to $s$.  

%%%%%%%%%%%%%%%%%%%%%%%%%%%%%%%%%%%%%%%%%%%%%%%%%%%%%%%%%%%%%%%%% 
\begin{subsection}{$\s$-Algebraic elements in $\s$-extensions } 	
	In this subsection, we let $(L,\sigma)$ be an extension $\sigma$-field  of a $\sigma$-field $(K,\sigma)$. An element $a\in L$ is called \textit{$\s$-algebraic over $K$} (or \textit{algebraic over $(K,\sigma)$}) if  $a$ is a root of a polynomial $P(T)\in K[T;\sigma]$. The following lemma gives a criterion for an element $a\in L$ to be algebraic over $(K,\sigma)$.
	\begin{lemma}\label{(L)Criterionalgebraic}
		For any element $a\in L$, the following are equivalent:\\
		(1) $a$ is $\s$-algebraic over $K$.\\
		(2) The $K$-vector subspace of $L$ generated by the elements $N_i(a)$, $i\geq 0$, is finite-dimensional. \\
		(3) There exists a finite-dimensional $K$-vector subspace $V$ of $L$ such that $1\in V$ and $a\s(V)\subset V$.   
	\end{lemma}
	\begin{proof}
		(1)$\implies$(2): Let $a\in L$ be $\s$-algebraic over $K$. Then, we have
		$$c_0+c_1N_1(a)+\cdots+c_{n}N_{n}(a)=0,$$
		for some $c_0,...,c_n\in K$ with $c_n\neq 0$. Applying $\s$ to this equation and then multiplying  with $a$, we obtain
		$$\sigma(c_0)a+\sigma(c_1)a\s(N_1(a))+\cdots+\s(c_{n})a\s(N_{n}(a))=0.$$
		Since $a\s(N_i(a))=N_{i+1}(a)$, this equation can be rewritten as
		$$\sigma(c_0)N_1(a)+\sigma(c_1)N_2(a)+\cdots+\s(c_{n})N_{n+1}(a)=0,$$
		proving that $N_{n+1}(a)$ is a $K$-linear combination of the elements
		$$N_0(a),N_1(a),\dots,N_{n}(a).$$
		Inductively, one can see that this holds for all $N_i(a)$, where $i\geq n+1$. In particular, the $K$-vector subspace of $L$ generated by the elements $N_i(a)$, $i\geq 0$, is finite-dimensional.\\
		(2)$\implies$(3): We show that the finite-dimensional $K$-vector subspace $V$ of $L$ generated by $N_i(a)$, $i\geq 0$, satisfies the desired conditions. Clearly, $1=N_0(a)\in V$. Using the formula $N_{i+1}(a)=a\s(N_i(a))$, one can easily see that  $a\s(V)\in V$, and we are done.\\
		(3)$\implies$(1): Let  $V$ be a finite-dimensional $K$-vector subspace  of $L$ which contains $1$ and satisfies $a\s(V)\subset V$. Let $n$ be the dimension of $V$. Since $1\in V$ and  $a\s(V)\subset V$, the identity $N_{i+1}(a)=a\s(N_i(a))$ implies that  $N_i(a)\in V$ for all $i\geq 0$. In particular, $N_0(a),...,N_{n}(a)$ are linearly dependent over $K$, that is,  $a$ is $\s$-algebraic over $K$. 
	\end{proof}	
Now, we give some applications of this lemma. 
	\begin{proposition}
		Let $b\in L$. There is $x\in L^*$ such that $\s(x)bx^{-1}$ is algebraic over $(K,\s)$ iff there exists a nonzero finite-dimensional $K$-vector subspace $W$ of $L$ such that  $b\s(W)\subset W$.
	\end{proposition}
	\begin{proof}
		Let $a=\s(x)bx^{-1}$ be algebraic over $(K,\s)$ for some $0\neq x\in L$. By Lemma \ref{(L)Criterionalgebraic}, there exists a finite-dimensional $K$-vector subspace $V$ of $L$ such that $1\in V$ and $a\s(V)\subset V$. We have
		$$\s(x)bx^{-1}\s(V)\subset V\implies b\s(xV)\subset xV.$$
		Therefore, the finite-dimensional $K$-vector space $W=xV$ satisfies the desired condition. 
		
		Conversely, let $b\s(W)\subset W$ for some finite-dimensional $K$-vector subspace $W$ of $L$.  Since $W$ is nonzero, $W$ contains a nonzero element $x$. Setting $V=x^{-1}W$, we have
		$$(\s(x)bx^{-1})V=(\s(x)bx^{-1})\s(x^{-1}W)=x^{-1}b\s(W)\subset x^{-1}W=V.$$
		Since $1\in V=x^{-1}W,$ Lemma \ref{(L)Criterionalgebraic} implies that $\s(x)bx^{-1}$ is algebraic over $(K,\s)$  and we are done. 
	\end{proof}
	\begin{proposition}
	If $a,b\in L$ are $\s$-algebraic over $K$, then $\s(a)$ and $ab$ are $\s$-algebraic over $K$ too.	
	\end{proposition}
	\begin{proof}
		 By Lemma \ref{(L)Criterionalgebraic}, there exist finite-dimensional $K$-vector subspaces $V$ and $W$ of $L$ such that $1\in V\cap W$, $a\s(V)\subset V$ and  $a\s(W)\subset W$. It is easy to see that the set
		 $$ VW=\{xy|x\in V,\, y\in W\},$$
		is a finite-dimensional $K$-vector subspace of $L$. Clearly, we have $1\in VW$. Moreover, 
		$$(ab)\s(VW)=(a\s(V))(b\s(W))\subset VW.$$
		It follows from Lemma \ref{(L)Criterionalgebraic} that  $ab$ is $\s$-algebraic over $K$. To prove that $\s(a)$ is $\s$-algebraic over $K$, let $x_1=1,x_2,...,x_n$ be a basis for the vector space $V$. Consider the vector space
		$$V'=K\s(x_1)+K\s(x_2)+\cdots+K\s(x_n).$$
		Clearly, $V'$ is finite-dimensional over $K$ and $1\in V'$. Since $\s(V)\subset V'$, we see that
		$$\s(a)\s(V')=\s(aV')=\s(K(a\s(x_1))+\cdots+K(a\s(x_n)))\subset \s(V)\subset V'.$$
		By Lemma \ref{(L)Criterionalgebraic}, $a$ is $\s$-algebraic over $K$.
	\end{proof}
	\begin{proposition}
		Let $a\in L$ be $\s$-algebraic over $K$ and $V$ be the $K$-vector subspace  of $L$ generated by $N_i(a)$, where $i\geq 0$. Then any $b\in L$ which is $\s$-conjugate to $a$ over $V$ is  $\s$-algebraic over $K$.
	\end{proposition}
	\begin{proof}
		By Lemma \ref{(L)Criterionalgebraic}, $V$ is finite-dimensional over $K$, and $a\s(V)\subset V$. Let $b\in L$ be $\s$-conjugate to $a$ over $V$. Then, there exists $x\in V^*$ such that 
		$b=\s(x)ax^{-1}$. Consider the vector space
		$$x^{-1}V=\{x^{-1}y\,|\, y\in V\}.$$
		Clearly, $V$ is a finite-dimensional  vector space over $K$. Since, $x\in V$, we have $1\in V$. Moreover, we have 
		$$b\s(x^{-1}V)=\s(x)ax^{-1}\s(x^{-1}V)=x^{-1}a\s(V)\subset x^{-1}V.$$
		By Lemma \ref{(L)Criterionalgebraic}, $b$ must be $\s$-algebraic over $K$.
	\end{proof}
	The sum of two $\s$-algebraic elements need not be  $\s$-algebraic, as the following example shows. 
	\begin{example}
		Consider the $\s$-field $(\mathbb{Q}(x),\s)$, where $\s$ is the unique automorphism of $\mathbb{Q}(x)$ determined by $\s(x)=x^{-1}$. Considering the $\s$-field $(\mathbb{Q}(x),\s)$ as an $\s$-extension  of the trivial $\s$-field $(\mathbb{Q},1_\mathbb{Q})$, we see that $x$ is $\s$-algebraic over $\mathbb{Q}$ because it is a root of the polynomial $T^2-1\in \mathbb{Q}[T;1_\mathbb{Q}]$. It is easy to show that
		$$N_i(x+1)=\frac{(x+1)^i}{x^{[\frac{i}{2}]}},$$
		for all $i\geq 0$. In particular, the elements  $N_i(x+1)$, $i\geq 0$, generate an infinite-dimensional $\mathbb{Q}$-vector subspace of $\mathbb{Q}(x)$. Using Lemma  \ref{(L)Criterionalgebraic}, 
		 we conclude that,  the element $x+1$, which is the sum of the $\s$-algebraic elements $x$ and $1$, is not $\s$-algebraic over $(\mathbb{Q},1_\mathbb{Q})$. 		  
	\end{example}	
	
	Using Lemma \ref{(L)ProductPQ(a)} and the fact that $K[T;\sigma]$ is a left principal ideal domain, we see that  if $a\in L$ is algebraic over $(K,\s)$, then there exists a unique  monic skew polynomial $P_a(T)\in K[T;\sigma]$ such that
	$$\{Q(T)\in K[T;\sigma]\,|\, Q(a)=0 \}=K[T;\sigma]P_a(T).$$
	The skew polynomial $P_a(T)$ is called the \textit{minimal (skew) polynomial} of $a$ over $(K,\sigma)$. 
	In general, minimal polynomials need not be irreducible. For example, the minimal polynomial of $i\in\mathbb{C}$ in the $\s$-extension  $(\mathbb{C},\,\bar\, \,)$ of the  $\s$-field $(\mathbb{R},1_\mathbb{R})$ is $T^2 -1$ which is not irreducible in $\mathbb{R}[T;1_\mathbb{R}]$. However, we can show that every element has a $\s$-conjugate whose minimal polynomial is irreducible. Here are the details:  
	As before, let $(L,\s)$ be an extension $\sigma$-field  of $(K,\sigma)$ and $a\in L$ be $\s$-algebraic over $K$. Assume that the minimal polynomial $P(T)=P_a(T)$ of $a$ over $(K,\s)$ is reducible. It follows that there exists $Q(T),R(T)\in K[T;\sigma]$ of degree $>0$ such that 
	$$P(T)=Q(T)R(T).$$
	Since $R(a)\neq 0$, Lemma \ref{(L)ProductPQ(a)} implies that $Q(a_1)=0$ where 
	$$a_1=\sigma(R(a))\, a\, R(a)^{-1}\in K\langle a \rangle.$$ 
	Here, $K\langle a \rangle$ is the $\s$-subfield of $L$ which is generated by $K$ and $a$. In particular, $a_1$ is $\sigma$-conjugate to $a$ over $K\langle a \rangle$, and $a_1$ is $\s$-algebraic over $K$. Let $P_1(T)\in K[T;\sigma]$  be the minimal  skew polynomial of $a_1$ over $K$. Note that $\deg P_1<\deg P$ and $Q(T)=Q_1(T)P_1(T)$ for some $Q_1(T)\in K[T;\sigma]$. Continuing this process, we obtain a sequence of elements $a_1,a_2,...$ such that $a_{i+1}\in K\langle a_i \rangle$ is $\s$-conjugate to $a_i$ over $K\langle a_i \rangle$, and  $a_{i+1}$ is $\s$-algebraic over $K$ with minimal skew polynomial $P_{i+1}(T)\in  K[T;\sigma]$. Furthermore, 
	$$\deg P>\deg P_1> \deg P_2>...$$
	Therefore, after finitely many steps, we obtain an element $a_i$ whose minimal skew polynomial is irreducible over $(K,\s)$. Thus we have proved the following   
	\begin{proposition}
		Let $a\in (L,\s)$ be $\s$-algebraic over $(K,\s)$. Then, there exists $b\in  K\langle a\rangle$ such that (1) $b$ is $\sigma$-conjugate to $a$ over $K\langle a \rangle$, (2) $b$ is $\s$-algebraic over $(K,\s)$ and the minimal skew polynomial of $b$ over $K$ is irreducible, and (3) $P_a(T)=Q(T)P_b(T)R(T)$ for some $Q(T),R(T)\in K[T;\sigma]$.
	\end{proposition}	

	A $\s$-extension $(L,\s)$ of a $\s$-field $(K,\s)$ is called \textit{$\s$-algebraic} if, as a $\s$-field, $(L,\s)$ is generated by a family of its elements each of which is $\s$-algebraic over $K$. It is called \textit{$\s$-simple} if $L=K\langle a\rangle$, where $a$ is $\s$-algebraic over $K$. Following \cite{treur1989separate}, we say that two elements $a\in L_1$ and $b\in L_2$ in $\s$-extensions $(L_1,\s)$ and $(L_2,\s)$ of $(K,\s)$  have the same \textit{$(K,s)$-type }(or simply \textit{$K$-type}) if there is an isomorphism $K\langle a \rangle\to K\langle b \rangle$ of $\s$-fields over  $(K,\s)$ which sends $a$ to $b$. The proof of the following lemma is left to the reader. 
	\begin{lemma}\label{SameKtype}
		(1) Let $a,b$ have the same $K$-type. Then, $a$ is $\s$-algebraic over $K$ iff $b$ is $\s$-algebraic over $K$. \\
		(2) If $a$ and $b$ are $\s$-algebraic of the same $K$-type, then they have the same minimal polynomial over $(K,\s)$.
	\end{lemma}
	In contrast to ordinary field extensions, the converse of the second statement in this lemma does not hold in the case of $\s$-algebraic elements. More precisely, it may happen that two  $\s$-algebraic elements have the same minimal polynomial, but are not of the same type. As an example, the reader can verify that the element $i$ in  the $\s$-extension  $(\mathbb{C},\,\bar\, \,)$ of the trivial $\s$-field $(\mathbb{R},1_\mathbb{R})$, and the element $x$ in the  $\s$-extension  $(\mathbb{R}(x),\s)$, where $\s(x)=1/x$, of the trivial $\s$-field $(\mathbb{R},1_\mathbb{R})$ have the same minimal polynomial $T^2-1$ over $(\mathbb{R},1_\mathbb{R})$. However, $i$ and $x$ are not of the same $\mathbb{R}$-type. 

	We conclude this subsection with a general observation. Let $(L,\s)$ be a $\s$-extension of $(K,\s)$. For an arbitrary element $a\in L$, there is a unique $\s$-homomorphism $$(K\{x\},\s)\to K[a]_\s,$$ which acts as the identity map on $K$ and sends $x$ to $a$. Clearly, this map is onto. So, letting $J_a$ be the kernel of this map, we obtain a  $\s$-isomorphism
	$$\psi_a:(\frac{K\{x\}}{J_a},\s)\to K[a]_\s,$$
	which through localization gives rise to a $\s$-isomorphism of $\s$-fields: 
	$$\phi_a:(Frac(\frac{K\{x\}}{J_a}),\s)\to K\langle a \rangle.$$
	Keeping the above notations, we prove the following:
	\begin{proposition}
		Let $(L_1,\s_1)$ and $(L_2,\s_2)$ be $\s$-extensions of $(K,\s)$. Then, $a\in L_1$ and $b\in L_2$ are of the same $K$-type iff $J_a=J_b$.  
	\end{proposition}
	\begin{proof}
		If $J_a=J_b$, then $\phi_b\phi_a^{-1}:K\langle a \rangle\to K\langle b \rangle$ is a $\s$-isomorphism which acts as the identity map on $K$ and sends $a$ to $b$. Therefore, $a$ and $b$ are of the same $K$-type. Conversely, let $a$ and $b$ have the same $K$-type. Then, there is a   $\s$-isomorphism $\phi:K\langle a \rangle\to K\langle b \rangle$  which acts as the identity map on $K$ and sends $a$ to $b$. Then, the $\s$-isomorphism 
		$$\phi_b^{-1}\phi \phi_a: (Frac(\frac{K\{x\}}{J_a}),\s)\to(Frac(\frac{K\{x\}}{J_b}),\s),$$
		acts as the identity map on $K$ and sends $x+J_a$ to $x+J_b$, from which it follows that $J_a=J_b$.
	\end{proof}
\end{subsection}
	
%%%%%%%%%%%%%%%%%%%%%%%%%%%%%%%%%%%%%%%%%%%%%%%%%%%%%%%%%%%%%%%%% 
%%%%%%%%%%%%%%%%%%%%%%%%%%%%%%%%%%%%%%%%%%%%%%%%%%%%%%%%%%%%%%%%%  

\begin{subsection}{Minimal polynomials} 
	
	In this part, we turn to the question of deciding when a given skew polynomial can be a minimal polynomial. In what follows, let $(K,\sigma)$ be a  $\sigma$-field, and 
	$$P(T)=T^{n+1}-c_{n}T^{n}-\cdots-c_1T-c_0\in K[T;\sigma],$$
	be a (monic) polynomial of degree $n+1>0$. Although, most of the results of this part hold if $P(T)\notin T\,K[T;\sigma]$, we assume, for simplicity, that $P(0)\neq 0$, that is, $c_0\neq 0$ (see Remark \ref{(R)GeneralP}).  
	
	Consider the ring $K[y_0,...,y_{n-1}]$  of polynomials over $K$ in $n$ (commuting) indeterminates, and  let $V_n$ be the set of all $Y\in K[y_0,...,y_{n-1}]$ of the form
	$$Y=d_0+d_{1}y_0+d_{2}y_0y_1+\cdots+d_{n-1}y_{0}y_{1}\cdots y_{n-2}+d_{n}y_{0}y_{1}\cdots y_{n-1},$$
	where $d_0,...,d_{n}\in K$. Clearly, $V_n$ is an $(n+1)$-dimensional $K$-vector subspace of  $K[y_0,...,y_{n-1}]$. First, we define a map $L:V_n\to V_n$ as follows: For
	$$Y=d_0+d_{1}y_0+\cdots+d_{n}y_{0}y_{1}\cdots y_{n-1}\in V_n,$$
	we set
	$$L(Y)=
	\sigma (d_{n}) c_0+(\sigma (d_{0})+\sigma (d_{n})c_{1})y_0+\cdots+(\sigma (d_{n-1})+\sigma (d_{n})c_{n})y_0y_{1}\cdots y_{n-1}.
	$$	
%	$$L(Y)=
%	\sigma (d_{n}) c_0+(\sigma (d_{0})+\sigma (d_{n})c_{1})y_0+(\sigma (d_{1})+\sigma (d_{n})c_{2})y_0y_1+\cdots+
%	$$
%	$$
%	(\sigma (d_{n-2})+\sigma (d_{n})c_{n-1})y_0y_{1}\cdots y_{n-2}+(\sigma (d_{n-1})+\sigma (d_{n})c_{n})y_0y_{1}\cdots y_{n-1}.
%	$$
	Next, we define a ring homomorphism $\phi:K[y_0,...,y_{n-1}]\to K(y_0,...,y_{n-1})$ as follows:
	By the universal property of $K[y_0,...,y_{n-1}]$, there exists a unique ring homomorphism
	$$\phi:K[y_0,...,y_{n-1}]\to K(y_0,...,y_{n-1}),$$
	such that $\phi(a)=\s(a)$, for all $a\in K$, and $$\phi(y_i)=\begin{cases}
		y_{i+1} & i=0,1,...,n-2,\\
		\frac{L(y_0y_1\cdots y_{n-1})}{y_0y_1\cdots y_{n-1}} & i=n-1.
	\end{cases}
	$$
	The following lemma gives some  properties of the maps $L$ and $\phi$. 
	\begin{lemma}\label{(L)ZerosofL}
		(1) For  $Y\in V_n$, $L(Y)=0$ iff $Y=0$. \\
		(2) The ring homomorphism $\phi$ is injective.\\
		(3) For all $Y\in V_n$ and $i>0$, $L^i(Y)=y_0\phi(y_0)\cdots \phi^{i-1}(y_{0})\phi^i(Y)$. 
	\end{lemma}
	\begin{proof}
		(1) Clearly, $L(0)=0$. Assume that $L(Y)=0$ for some nonzero
		$$Y=d_0+d_{1}y_0+\cdots+d_{n}y_{0}y_{1}\cdots y_{n-1}\in V_n.$$
		If $d_n=0$, then $L(Y)=0$ implies that 
		$$\sigma(d_0)=\sigma(d_1)=\cdots=\sigma(d_{n-1})=0\implies Y=0.$$
		If $d_n\neq 0$,  then $\s(d_n)c_0=0$ implies that  $c_0=0$, a contradiction. This completes the proof of (1). \\
		(2) Suppose, on the contrary, that $\phi$ is not injective. Then, there exist ordinary polynomials 
		$$f_0(y_0,...,y_{n-2}),\dots, f_m(y_0,...,y_{n-2})\in K[y_0,...,y_{n-2}],$$
		such that $f_0f_m\neq 0$ and 
		$\phi(\sum_{i=0}^mf_i(y_0,...,y_{n-2})y_{n-1}^i)=0.$
		Computing the left-hand side using the definition of $\phi$, we obtain
		\begin{equation}\label{(E)kernelofphi}		
			\sum_{i=0}^mg_i(y_1,...,y_{n-1})\left( \frac{L(y_0y_1\cdots y_{n-1})}{y_0y_1\cdots y_{n-1}}\right) ^i=0,	
		\end{equation}
		where $g_0,...,g_m\in K[y_0,...,y_{n-2}],$ with $g_0g_m\neq 0$. Writing
		\begin{align*}
			\frac{L(y_0y_1\cdots y_{n-1})}{y_0y_1\cdots y_{n-1}} = & \frac{c_0+c_{1}y_0+\cdots+c_{n-1}y_{0}y_{1}\cdots y_{n-2}+c_{n}y_{0}y_{1}\cdots y_{n-1}}{y_0y_1\cdots y_{n-1}}\\
			=& c_0y_0^{-1}y_1^{-1}\cdots y_{n-1}^{-1}+c_1y_1^{-1}\cdots y_{n-1}^{-1}+\dots+c_{n-1}y_{n-1}^{-1}+c_n
		\end{align*}
		and expanding the left-hand side of Equation \ref{(E)kernelofphi} as a polynomial in $y_0^{-1}$, we see that 
		the highest coefficient of $y_0^{-1}$ in the left-hand side of the equation  is equal to
		$$(c_0y_1^{-1}\cdots y_{n-1}^{-1})^mg_m(y_1,...,y_{n-1})\neq 0,$$
		a contradiction. This completes the proof of (2).\\
		(3) We first note that the homomorphism $\phi$ can be extended to an endomorphism $\phi:K(y_0,...,y_{n-1})\to K(y_0,...,y_{n-1}),$ by Part (2). In particular, the formula given in (3) is well-defined. Given $$Y=d_0+d_{1}y_0+d_{2}y_0y_1+\cdots+d_{n}y_{0}y_{1}\cdots y_{n-1}\in V_n,$$
		we write
		\begin{align*}
			\phi\left(Y\right) = & \sigma (d_{0})+\sigma (d_{1})y_1+\sigma (d_{2})y_1y_2+\cdots+\sigma (d_{n})y_1y_2\cdots y_{n-1}y_n\\
			=& \sigma (d_{0})+\sigma (d_{1})y_1+\sigma (d_{2})y_1y_2+\cdots+\sigma (d_{n-1})y_1y_2\cdots y_{n-1}\\
			&+\sigma (d_{n})\frac{c_0+c_{1}y_0+\cdot+c_{n}y_{0}y_{1}\cdots y_{n-1}}{y_0}
		\end{align*}
		This gives 
		\begin{align*}
			y_0\phi\left(Y\right) = & \sigma (d_{0})y_0+\sigma (d_{1})y_0y_1+\cdots+\sigma (d_{n-1})y_0y_1y_2\cdots y_{n-1}\\
			& + \sigma (d_{n})\left( c_0+c_{1}y_0+c_{2}y_1y_2+c_{n}y_{0}y_{1}\cdots y_{n-1}\right)\\
			=& \sigma (d_{n}) c_0+(\sigma (d_{0})+\sigma (d_{n})c_{1})y_0+\cdots+(\sigma (d_{n-1})+\sigma (d_{n})c_{n})y_0y_{1}\cdots y_{n-1}\\
			=& L(Y)
		\end{align*}
		proving (3) for $i=1$. Now, a simple induction yields the desired formula. 
	
	\end{proof}
	  
	Let $S$ be the multiplicative subset of $K[y_0,...,y_{n-1}]$ which is generated by all the nonzero elements of $V_n$. It follows from  Lemma \ref{(L)ZerosofL} that $\phi$, through localization, gives rise to a ring homomorphism
 	$$\s_P:S^{-1}K[y_0,...,y_{n-1}]\to S^{-1}K[y_0,...,y_{n-1}], \,\s_P(rs^{-1})=\phi(r)\phi(s)^{-1},$$
 	where $r\in K[y_0,...,y_{n-1}]$ and $s\in S$. Using Part (2) of Lemma \ref{(L)ZerosofL}, we see that $\s_p$ is injective. Therefore, the pair
 	$$(K[P],\s_P):=(S^{-1}K[y_0,...,y_{n-1}],\s_P),$$
 	is a $\s$-domain. The $\s$-field of fractions of $(K[P],\s_P)$ is denoted by  $(K(P),\s_P)$. Note that  $(K(P),\s_P)$ is a $\s$-extension of $(K,\s)$. Now, we show that  $P(T)$ is the minimal polynomial of an element in $(K(P),\s_P)$.
	\begin{proposition}\label{(P)rootinK(P)}
		The element  $y_0\in K(P)$ is a nonzero root of $P(T)$ in the $\s$-extension  $(K(P),\s)$ of $(K,\s)$. Moreover, the minimal skew polynomial of  $y_0\in K(P)$ over $(K,\s)$ is $P(T)$.        
	\end{proposition}
	\begin{proof}
		The fact that $y_0\in K(P)$ is a root of $P(T)$ follows from the fact that
		$$P(y_0)=y_{0}y_{1}\cdots y_{n-1}	\frac{L(y_0y_1\cdots y_{n-1})}{y_0y_1\cdots y_{n-1}}-c_{n}y_{0}y_{1}\cdots y_{n-1}-\cdots-c_{1}y_0-c_{0}=0.$$
		 If $y_0$ is a root of a skew polynomial 
		$$d_0+d_1T+\cdots+d_mT^m\in K[T;\sigma],$$
		where $m<\deg P$, then 
		$$0=d_0+d_1y_0+d_2y_0y_1+\cdots+d_my_0y_1\cdots y_{m-1},$$
		which implies $d_0=...=d_{m-1}=0$, since $y_0,...,y_{n-1}$ are algebraically independent over $K$. This completes the proof. 
	\end{proof}
	Next, we show that the $\s$-domain $(K[P],\s_P)$ has the following universal property:
	\begin{theorem}\label{(T)K[P]universal}
		
		 An element $0\neq a$ in a $\s$-extension $(L,\s)$ of $(K,\s)$ has $P(T)$ as its minimal polynomial over $(K,\s)$ iff there exists a  $\s$-homomorphism  $$\psi_a:(K[P],\sigma_P)\to   (L,\s),$$ over $(K,\s)$, such that $\psi_a(y_0)=a$. Moreover, such a $\s$-homomorphism is unique.  
		                	
	\end{theorem}
	\begin{proof}
		First, suppose that $P(T)$ is the minimal polynomial of $a\in L$. By the universal property of  $K[y_0,...,y_{n-1}]$, there exists a unique $K$-algebra homomorphism 
		$$\psi:	K[y_0,...,y_{n-1}]\to L$$ such that 
		$\psi(y_i)=\s^i(a)$ for $i=0,1,...,n-1$.
		The condition $P(a)=0$ gives
		$$\s^{n}(a)\cdots\s(a)a
		=c_{0}+c_{1}a+c_{2}a\s(a)+\cdots+c_{n}a\s(a)\s^2(a)\cdots \s^{n-1}(a),
		$$
		from which it follows that $$\s^{n}(a)\psi(y_{n-1})\cdots \psi(y_0)=\psi(L(y_0,...,y_{n-1})).$$  Localizing $\psi$ at $L(y_0,...,y_{n-1})$, we obtain 
		a $K$-algebra homomorphism 
		$$\psi_1:	K[y_0,...,y_{n-1}][L(y_0,...,y_{n-1})^{-1}]\to L,$$
		such that $\psi_1(\frac{L(y_0y_1\cdots y_{n-1})}{y_0y_1\cdots y_{n-1}})=\s^{n}(a)$. For any 
		$$0\neq r=d_0+d_1y_0+d_2y_0y_1+\cdots+d_ny_0y_1\cdots y_{n-1}\in  S,$$
		we have
		$$\psi_1(r)=d_0+d_1a+d_2a\s(a)+\cdots+d_na\s(a)\cdots \s^{n-1}(a)\neq0,$$
		since the minimal polynomial of $a$ has degree $n+1$. Therefore, we can localize  $\psi_1$ at $S$ to obtain 
		a $K$-algebra homomorphism 
		$$\psi_a:	S^{-1} K[y_0,...,y_{n-1}]\to L.$$ 
		Since $\psi_a(\s_P(y_i))=\s(\psi_a(y_i))$ for $i=0,1,...,n-1$, one can easily verify that $\psi_a$ is a $\s$-homomorphism satisfying $\psi_a(y_0)=a$. Conversely, suppose that 
		$$\psi_0:(K[P],\sigma_P)\to   (L,\s),$$
		is a $\s$-homomorphism  over $(K,\s)$ which satisfies $\psi_0(y_0)=a$. We have $P(a)=0$ since $\psi_0$ is a $\s$-homomorphism, and $P(y_0)=0$ by Proposition \ref{(P)rootinK(P)}. The fact that the image of any element in $S$ is invertible in $L$ implies that $P(T)$ is the minimal polynomial of $a$. The uniqueness part follows from the fact that any $\s$-homomorphism on $(K[P],\s_P)$ is completely determined  by its value at  $y_0$ and its restriction to $K$. 
	\end{proof}
	Now, we give some applications of this theorem and Proposition \ref{(P)rootinK(P)}. The following corollary is an immediate consequence of  Proposition \ref{(P)rootinK(P)}. 
	\begin{corollary}
		Any nonconstant monic polynomial $Q(T)\in K[T;\s]$ satisfying $Q(0)\neq 0$ is the minimal polynomial of a (nonzero) element in some $\s$-extension of $(K,\s)$.
	\end{corollary}
		The following proposition, the proof of which is left to the reader, explains the connection between $K$-types and prime $\s$-ideals.  
	\begin{proposition}\label{(P)correspondenceKtype}
		Let $Q(T)\in K[T;\s]$ be a nonconstant monic polynomial $Q(T)\in K[T;\s]$ satisfying $Q(0)\neq 0$. For every prime $\s$-ideal $I$ of $(K[Q],\s_Q)$, the polynomial $Q(T)$ is the minimal polynomial of the element $y_0+I$ in the $\s$-field $(Frac(K[Q]/I),\s)$ of fractions of $(K[Q]/I,\s)$. Moreover,  the following holds:\\
		(1) $(Frac(K[Q]/I),\s)=K\langle y_0+I\rangle$, \\
		(2) For any element $0\neq a\in L$ in a $\s$-extension of $(K,\s)$ whose minimal polynomial over $K$ is $Q(T)$, there is a unique prime $\s$-ideal $I_a$ of $(K[Q],s_Q)$ such that $y_0+I_a$ has the same $K$-type as $a$.    
	\end{proposition}
%proof: The elements of S are invertible in K[Q]. 
	The assignment $a\mapsto I_a$, given in this proposition, establishes a 1-1 correspondence between distinct $K$-types with a fixed minimal polynomial $Q(T)\in K[T;\s]$, and the set of prime $\s$-ideals of $(K[Q],s_Q)$. We study this correspondence more closely. Let $Q(T)\in K[T;\s]$ be a nonconstant monic polynomial $Q(T)\in K[T;\s]$ satisfying $Q(0)\neq 0$. Let $\mathcal{P}(Q)$ be the set of all prime  $\s$-ideals of $(K[Q],s_Q)$. Let $\mathcal{L}(Q)$ be the class of all elements in $\s$-extensions of $(K,\s)$ whose minimal polynomial is $Q(T)$. Consider the equivalence relation $\sim$ on $\mathcal{L}(Q)$ for which $a\sim b$ iff $a$ and $b$ have the same $K$-type. Let $\mathcal{T}(Q)$ be the class of equivalence classes of $\sim$ in $\mathcal{L}(Q)$. Then, the assignment
	$I\mapsto [y_0+I]$, where $[a]$ denotes the equivalence class of $a\in \mathcal{L}(Q)$, establishes a 1-1 correspondence between  $\mathcal{P}(Q)$ and $\mathcal{T}(Q)$. In particular, $\mathcal{T}(Q)$  is a set, called the set of \textit{$K$-types of $Q(T)$}. Using this correspondence, we can define a relation $\leq $ on $\mathcal{T}(Q)$: Given $[a],[b]\in \mathcal{L}(Q)$, we write $[a]\leq [b]$ if $I_a\subset I_b$. There is a unique minimal element in $\mathcal{T}(Q)$ which corresponds to the prime $\s$-ideal $(0)$ of $(K[Q],s_Q)$. Also, using Zorn's lemma, one can show that for any $[a]\in \mathcal{T}(Q)$, there exists a maximal element $[b]\in \mathcal{T}(Q)$ such that $[a]\leq [b]$. We note that the set $\mathcal{P}(Q)$  is a subset of the affine scheme $Spec(K[y_0,...,y_{n-1}])$, where $n=\deg Q-1$. We conclude this part with an example and a remark: 

		\begin{example}
			Let $K$ be an arbitrary (commutative) field. Consider the trivial $\s$-field $(K,1_K)$. Let $P(T)\in K[T;1_K]$ be the polynomial
			$$P(T)=T^2-1.$$
			One can easily check that $K[P]=S^{-1}K[y_0]$, where $S$ is the multiplicative subset of $K[y_0]$ which is generated by $y_0-a$, where $a\in K$. Moreover, we have $\s_P(y_0)=y_0^{-1}$. 
			Since every prime ideal of $K[P]$, apart from the zero ideal, is of the form $K[P]\,f(y_0)$ where $f(y_0)\in K[y_0]$ is an irreducible polynomial of degree $>1$, it is easy to see that a prime ideal $K[P]\,f(y_0)$ of  $K[P]$ is  a $\s$-ideal iff there exists an element $a\in K$ such that
			$$y_0^nf(\frac{1}{y_0})=af(y_0),$$ where $n$ is the degree of $f$. In other words, $K[P]\,f(y_0)$ is a prime $\s$-ideal iff $f_0(y)$ is irreducible and palindromic of degree $>1$.  Therefore, there is a 1-1 correspondence between the set of monic palindromic irreducible polynomials $f(y_0)\in K[y_0]$ of degree $1$ and the set  $\mathcal{P}(T^2-1)\setminus \{(0)\}$. In particular, if $K$ is an algebraically closed field, we have $\mathcal{P}(T^2-1)= \{(0)\}$. We note that the $\s$-extension corresponding to the zero ideal $(0)$ is $(K(y_0),\s)$ with $\s(y_0)=y_0^{-1}$.   
		\end{example}     
		
		\begin{remark}\label{(R)GeneralP}
		Results similar to the ones presented in this part hold for any monic polynomial not contained in $T\, K[T;\s]$. In fact, one can show that a monic polynomial   $P(T)\in K[T;\s]$ is the minimal polynomial of a (nonzero) element in some $\s$-extension of $(K,\s)$  iff $P(T)\notin T\, K[T;\s]$. Also,  a skew polynomial $P(T)\in K[T;\sigma]$ has a nonzero root in some $\s$-extension  of $(K,\s)$ iff $P(T)$ is not of the form $aT^m$ where $a\in K^*$ and $m\geq 0$. Note that the condition  $P(T)\notin T\, K[T;\s]$ is equivalent to the condition $P(0)\neq 0$ provided that $(K,\s)$ is inversive. 
		\end{remark}

\end{subsection}
%%%%%%%%%%%%%%%%%%%%%%%%%%%%%%%%%%%%%%%%%%%%%%%%%%%%%%%	
%%%%%%%%%%%%%%%%%%%%%%%%%%%%%%%%%%%%%%%%%%%%%%%%%%%%%%%%%%%%%%%%%
%%%%%%%%%%%%%%%%%%%%%%%%%%%%%%%%%%%%%%%%%%%%%%%%%%%%%%%%%%%%%%%%% 

\end{section} 
%%%%%%%%%%%%%%%%%%%%%%%%%%%%%%%%%%%%%%%%%%%%%%%%%%%%%%%%%%%%%%%%
%%%%%%%%%%%%%%%%%%%%%%%%%%%%%%%%%%%%%%%%%%%%%%%%%%%%%%%%%%%%%%%%
%%%%%%%%%%%%%%%%%%%%%%%%%%%%%%%%%%%%%%%%%%%%%%%%%%%%%%%%%%%%%%%%
%%%%%%%%%%%%%%%%%%%%%%%%%%%%%%%%%%%%%%%%%%%%%%%%%%%%%%%%%%%%%%%%
%%%%%%%%%%%%%%%%%%%%%%%%%%%%%%%%%%%%%%%%%%%%%%%%%%%%%%%%%%%%%%%%
%%%%%%%%%%%%%%%%%%%%%%%%%%%%%%%%%%%%%%%%%%%%%%%%%%%%%%%%%%%%%%%%
%%%%%%%%%%%%%%%%%%%%%%%%%%%%%%%%%%%%%%%%%%%%%%%%%%%%%%%%%%%%%%%%
%%%%%%%%%%%%%%%%%%%%%%%%%%%%%%%%%%%%%%%%%%%%%%%%%%%%%%%%%%%%%%%%
%%%%%%%%%%%%%%%%%%%%%%%%%%%%%%%%%%%%%%%%%%%%%%%%%%%%%%%%%%%%%%%%
%%%%%%%%%%%%%%%%%%%%%%%%%%%%%%%%%%%%%%%%%%%%%%%%%%%%%%%%%%%%%%%%
%%%%%%%%%%%%%%%%%%%%%%%%%%%%%%%%%%%%%%%%%%%%%%%%%%%%%%%%%%%%%%%%
%%%%%%%%%%%%%%%%%%%%%%%%%%%%%%%%%%%%%%%%%%%%%%%%%%%%%%%%%%%%%%%%
%%%%%%%%%%%%%%%%%%%%%%%%%%%%%%%%%%%%%%%%%%%%%%%%%%%%%%%%%%%%%%%%
%%%%%%%%%%%%%%%%%%%%%%%%%%%%%%%%%%%%%%%%%%%%%%%%%%%%%%%%%%%%%%%%
%%%%%%%%%%%%%%%%%%%%%%%%%%%%%%%%%%%%%%%%%%%%%%%%%%%%%%%%%%%%%%%%
%%%%%%%%%%%%%%%%%%%%%%%%%%%%%%%%%%%%%%%%%%%%%%%%%%%%%%%%%%%%%%%%
\begin{section}{Algebraically closed $\s$-fields}	
	As is well-known, every (commutative) field has an algebraic closure which is unique up to isomorphism. This property of  fields is based on the following two facts: 
	\begin{itemize}
		\item[F1]  Any irreducible polynomial over a field has a root in some extension of the field.
		\item[F2] Any  polynomial over a field has a unique root up to isomorphism.
	\end{itemize}
	Thanks to F1, one can construct an algebraic closure of a given field by successively "adjoining" roots of irreducible polynomials.   The uniqueness of algebraic closure for fields is a consequence of F2.    
	
	More generally, the notion of root can be defined for other algebraic objects, using which one can define the corresponding concept of algebraic closedness. Generally speaking, if F1 holds for a particular algebraic object, then one can construct algebraic closures for the algebraic object under study by adjoining roots "one after another". However, if F2 does not hold, there may exist several notions of algebraic closedness.  We refer the reader to  \cite{cohn1975presentations,niven1941equations} for the case of skew fields, and \cite{neumann1943adjunction, scott1951algebraically} for the case of groups, (see also  \cite{robinson1971notion} for a model-theoretic approach to this problem).   In this section, we  investigate  the question of algebraic closedness for $\s$-fields. Note that F1 holds for skew roots, see Proposition \ref{(P)rootinK(P)}. However, F2 does not hold for skew roots, see  Proposition \ref{(P)correspondenceKtype}. The lack of F2 for $\s$-fields gives rise to different notions of algebraic closedness for $\s$-fields.

%%%%%%%%%%%%%%%%%%%%%%%%%%%%%%%%%%%%%%%%%%%%%%%%%%%%%%%%%%%%%%%%%	
	%%%%%%%%%%%%%%%%%%%%%%%%%%%%%%%%%%%%%%%%%%%%%%%%%%%%%%%%%%%%%%%%% 
	%%%%%%%%%%%%%%%%%%%%%%%%%%%%%%%%%%%%%%%%%%%%%%%%%%%%%%%%%%%%%%%%%  

	\begin{subsection}{Some notions of algebraically closed $\s$-fields} 
		The simplest notion of algebraic closedness for $\sigma$-field is the following: A $\s$-field $(L,\sigma)$ is called \textit{$1$-algebraically closed} if every  nonconstant  polynomial in $L[T;\s]$ has a root in $L$. In this part, we show that every $\s$-field has  a $\s$-extension which is $1$-algebraically closed. To prove this result, we implement a classical method used to prove the existence of algebraic closures for ordinary fields, see \cite[Theorem 2.5]{lang2012algebra}. 
		
		\begin{theorem}\label{(T)algebraicallyclosed}
			
			Every $\s$-field can be embedded in a $1$-algebraically closed $\s$-field. 
			
		\end{theorem}
		\begin{proof}
			Let $(K,\sigma)$ be an arbitrary $\sigma$-field. The idea is to adjoin skew roots of irreducible skew polynomials until there are no irreducible polynomials of degree $> 1$ left to deal with. This will be done in a series of steps as described below:\\
			\textbf{Selecting and ordering polynomials (Step 1)}:  Let $D$  be the set of all monic irreducible  polynomials in $K[T;\s]$ of degree $>1$ which do not have a root in $K$. 
			 Consider a total order on $D_1=\{1\}\cup D$ which turns $D_1$ into a well-ordered set with $1$ being its least element. Let $\alpha_0$ be the ordinal of $D_1$. Given an ordinal $\alpha\leq \alpha_0$, we denote the $\alpha$-th element of $D_1$ by $P_\alpha$.  Note that any polynomial in $D$ has a nonzero constant term since otherwise it would not be irreducible.\\ 
			 \textbf{Adjoining roots (Step 2)}: 
			 Using transfinite induction, we define a sequence $(K_\alpha,\s)$, where $\alpha\leq \alpha_0$ is an ordinal, of $\s$-fields as follows: we set $(K_0,\s)=(K,\s)$ and $K_\alpha=(\cup_{\beta<\alpha}K_\beta)(P_\alpha)$, see Proposition \ref{(P)rootinK(P)}. We note that   $(K_\alpha,\s)$ is a $\s$-extension of $(K_\beta,\s)$ if $\beta\leq \alpha\leq \alpha_0$. Now, we consider the $\s$-field $(K^{(1)},\s)=(K_{\alpha_0},\s)$. Clearly,  $(K^{(1)},\s)$ is   a $\s$-extension of $(K,\s)$.  By construction, every $P(T)\in D$ has a root in $K^{(1)}$. Therefore, any irreducible nonconstant polynomial in  $K[T;\s]$ has a root in  $K^{(1)}$. \\
			  \textbf{Induction (Step 3)}: 
			  Applying Steps 1 and 2 to $(K^{(1)},\s)$ gives rise to a   $\s$-extension   $(K^{(2)},\s)$ of  $(K^{(1)},\s)$ in which any nonconstant polynomial belonging to $K^{(1)}[T;\s]$ has a root. Inductively, we obtain a sequence of $\s$-fields 
			$$(K^{(1)},\s)\subset (K^{(2)},\s)\subset (K^{(3)},\s)\subset \dots,$$
			such that  any nonconstant polynomial belonging to  $K^{(n)}[T;\s]$  has a root in  $(K^{(n+1)},\s)$. Finally, we consider the $\s$-field
			$$(K^{\infty},\s)=\cup_{n=1}^\infty (K^{(n)},\s).$$
			Any  skew polynomial $P(T)\in K^{\infty}[T;\s]$ belongs to some 
			$K^{(n)}[T;\s]$. It follows that if $P(T)$  has a root in $(K^{(n+1)},\s)$, hence in $(K^{\infty},\s)$. Therefore, $(K^{\infty},\s)$ is $1$-algebraically closed and contains $(K,\s) $ as a $\s$-subfield. This completes the proof of the theorem. 
		\end{proof}

		A $1$-algebraically closed $\s$-field is only required to contain at least one root of every nonconstant skew polynomial. The requirement that any skew polynomial have more than one root, gives rise to the following notion of algebraic closedness: Let $c$ be a cardinal number. A $\s$-field $(L,\s)$ is called \textit{$c$-algebraically closed} if any skew polynomial over $(L,\s)$, of degree $>1$ and with a nonzero constant term, has at least $c$ distinct roots in $L$.  Clearly, if $c\leq c'$ are two cardinals, then any $c'$-algebraically closed $\s$-field is $c$-algebraically closed.

	\begin{theorem}\label{(T)$c$-algebraically closed}
	
	Let $c$ be a cardinal number. Any $\s$-field can be embedded in a $c$-algebraically closed $\s$-field. 
	
	\end{theorem}

	\begin{proof}
		The proof is similar to that of Theorem \ref{(T)algebraicallyclosed}. Here is an outline of the proof. The details are left to the reader to verify.\\   
		\textbf{Selecting and ordering polynomials (Step 1)}:  Let $D$  be the set of all monic polynomials in $K[T;\s]$ which are of degree $>1$ and have a nonzero constant term. 
		Consider a total order on $D_1=\{1\}\cup D$ which turns $D_1$ into a well-ordered set with $1$ being its least element. Let $\alpha_0$ be the ordinal of $D_1$. Given an ordinal $\alpha\leq \alpha_0$, we denote the $\alpha$-th element of $D_1$ by $P_\alpha$. \\ 
		\textbf{Adjoining roots (Step 2)}: 
		Using transfinite induction, we define a sequence $(K_\alpha,\s)$, where $\alpha\leq \alpha_0$ is an ordinal, of $\s$-fields as follows: We set$(K_0,\s)=(K,\s)$; $K_\alpha$ is obtained from $\cup_{\beta<\alpha}K_\beta$ by adjoining $c$ copies of roots of $P_\alpha$ using  Proposition \ref{(P)rootinK(P)}. In particular,
        $(K^{(1)},\s)=(K_{\alpha_0},\s)$ is  a $\s$-extension of $(K,\s)$, and,  by construction, every $P(T)\in D$ has at least $c$ distinct roots in $K^{(1)}$. \\
		\textbf{Induction (Step 3)}: 
		Using Steps 1, 2  and induction, we can construct a  sequence of $\s$-fields
		$$(K^{(1)},\s)\subset (K^{(2)},\s)\subset (K^{(3)},\s)\subset \dots,$$
		such that  any polynomial belonging to  $K^{(n)}[T;\s]$, of degree $>1$ and with a nonzero constant term, has a root in  $(K^{(n+1)},\s)$. One can see that the $\s$-field
		$$(K^{\infty},\s)=\cup_{n=1}^\infty (K^{(n)},\s).$$
		 is $c$-algebraically closed and contains $(K,\s) $ as a $\s$-subfield. 
	\end{proof}

	$\s$-field is $c$-algebraically closed, it is only enough to check the condition for polynomials of degree 2.

	In the proof of Theorem \ref{(T)$c$-algebraically closed}, we have used Proposition \ref{(P)rootinK(P)} as a consequence of which most roots would be of the same type. Taking into account that skew polynomials can have roots of different types, we introduce the following definition: A $\s$-field $(L,\sigma)$ is called \textit{type-algebraically closed} if for any element $a$ in some $\s$-extension of $(L,\sigma)$ with minimal polynomial $P(T)\in L[T;\s]$ over $(L,\sigma)$, there is $b\in L$ whose minimal polynomial over the $\s$-subfield $K$ of $(L,\sigma)$  generated by the coefficients of $P(T)$ is the polynomial $P(T)$, and $a$ and $b$ have the same $K$-type.   Using Proposition \ref{(P)correspondenceKtype} and modifying the steps in the proof Theorem \ref{(T)algebraicallyclosed}, one can prove the existence of type-algebraically closed $\s$-fields. More precisely, we have the following result the proof of which is left to the reader.
	
	\begin{theorem}\label{(T)type-algebraicallyclosed}
		
		Any $\s$-field can be embedded in a type-algebraically closed $\s$-field. 
		
	\end{theorem}

%	Any type-algebraically closed $\s$-field is $1$-algebraically closed, but a type-algebraically closed $\s$-field  may not be $2$-algebraically closed. 
	One can define a "hybrid version"  of algebraic closedness in which every polynomial has at least a prescribed number of distinct roots of every possible type. We leave it to the reader to formulate and prove the existence of such algebraically closed $\s$-fields. We conclude this  part with an observation regarding polynomial factorization in algebraically closed $\s$-fields.
	
	\begin{proposition}\label{(P)Factorizationinalgebraicallyclosed}
	(a) Let $(K,\s)$ be a $1$-algebraically closed $\s$-field. Then, any nonzero skew polynomial $P(T)\in K[T;\s]$ splits completely, that is, there exist $a,a_1,...,a_n\in K^*$ and $m\geq 0$ such that
	$$P(T)=a(T-a_n)\cdots(T-a_2)(T-a_1)T^m.$$
	(b) Let $(K,\s)$ be $\aleph_0$-algebraically closed. Then, any $P(T)\in K[T;\s]$ can be written as 
	$$P(T)=a(T-a_n)\cdots(T-a_2)(T-a_1)T^m,$$
	where $a\in K^*$, $a_1,...,a_n\in K^*$ are distinct, and $m\geq 0$.  
	\end{proposition}
	\begin{proof}
		Let $P(T)\in K[T;\s]$. There exists $m\geq 0$ and $P(T)\in K[T;\s]$ such that $Q(T)$ has a nonzero constant term and $P(T)=Q(T)T^m$. If $Q(T)$ is a constant, we are done. Otherwise, there exists $a_1\in K$ such that $Q(a_1)=0$ because $(K,\s)$ is $1$-algebraically closed. By Lemma \ref{(L)DivisionP(a)}, there exists $Q_1(T)\in K[T;\s]$ such that $$Q(a)=Q_1(a)(T-a_1).$$ Clearly, $Q_1(T)$ has a nonzero constant term. Now, we prove (a) and (b) as follows:\\
		(a) Using induction on the degree of $
		P$, one can obtain the desired factorization. \\
		(b) If $Q_1(T)$ is a constant, we are done. Otherwise, we can find a root $a_1\neq a_2\in K$ since $K$ is $\aleph_0$-algebraically closed. A simple induction on the degree of $Q(T)$ completes the proof. 
	\end{proof}

\end{subsection}
%%%%%%%%%%%%%%%%%%%%%%%%%%%%%%%%%%%%%%%%%%%%%%%%%%%%%%%%%%%%%%%%%	
%%%%%%%%%%%%%%%%%%%%%%%%%%%%%%%%%%%%%%%%%%%%%%%%%%%%%%%%%%%%%%%%% 
%%%%%%%%%%%%%%%%%%%%%%%%%%%%%%%%%%%%%%%%%%%%%%%%%%%%%%%%%%%%%%%%%  

\begin{subsection}{Conjugacy classes in algebraically closed $\s$-fields} 
	In order to determine whether a $\s$-field is $c$-algebraically closed for $c>1$, it is enough to study polynomials of degree two. More precisely, we have the following: 
	\begin{proposition}\label{(P)2-algebraic}
		Let $c>1$ be a cardinal number. A $\s$-field $(L,\s)$ is $c$-algebraically closed iff it is  $1$-algebraically closed and every skew polynomial $P(T)\in L[T;\s]$ of degree $2$ whose constant term is nonzero has at least $c$ distinct roots in $L$. 
	\end{proposition}
	\begin{proof}
		To prove the "if" direction,  let $P(T)\in L[T;\s]$ be an arbitrary polynomial satisfying $\deg P>1$ and $P(0)\neq 0$. Since $(L,\s)$ is  $1$-algebraically closed, it follows from Proposition \ref{(P)Factorizationinalgebraicallyclosed}, that there are skew polynomials $P_1(T),P_2(T)\in  L[T;\s]$ such that $\deg P_2=2$ and $P(T)=P_1(T)P_2(T)$. Since any root of $P_2(T)$ is a root of $P_1(T)$, we see that $P(T)$ has at least $c$ distinct roots in $L$. The other direction is trivial. 
	\end{proof}
	
	To present the next result, we need to review some facts regarding roots of skew polynomials in  a fixed conjugacy class. For more details see \cite{LamLeroy1988, lamleroy1988algebraic}. Let $(K,\s)$ be a $\s$-field with fixed field $F$.
	Let $P(T)=\sum_{i=0}^na_iT^i\in K[T;\s]$. 
	An element $\s(x)x^{-1}$, where $x\in K^*$, is a root  of $P(T)$ iff
	$$P(\s(x)x^{-1})=0\iff \sum_{i=0}^na_i\s^i(x)x^{-1}=0\iff \sum_{i=0}^na_i\s^i(x)=0.$$
	The set $E(P,1)$ of all $x\in K$ for which $\sum_{i=0}^na_i\s^i(x)=0$ is a vector space over $F$. One can see that the assignment $[x]\mapsto \s(x)x^{-1}$ establishes a bijection between the projective space $\mathbb{P}(E(P,1))$ over $F$ and the set of roots of 
	$P(T)$ in the conjugacy class $C(1)$ containing $1$. In particular, if $F$ and $\dim_F E(P,1)$ are finite,  the number of roots of $P(T)$ in $C(1)$ is equal to
	$$\frac{|F|^{\dim_F E(P,1)}-1}{|F|-1}.$$	
	\begin{theorem}\label{(T)conjugacyclasses3-algebraic}
		
		Let $(K,\s)$ be a $2$-algebraically closed $\s$-field with fixed field $F$. Then, $K$ is $3$-algebraically closed iff $K$ has exactly two distinct $\s$-conjugacy classes. Moreover, $K$ is $3$-algebraically closed, then $(K,\s)$ is $(1+|F|)$-algebraically closed. In particular, if $K$ is $3$-algebraically closed and  $F$ is infinite,  $(K,\s)$  is $\aleph_0$-algebraically closed. 
%		$$c=\begin{cases}
%		1+|F| &\text{ if } 	|F|<\infty \text{ and } \dim_F E(P,1)<1\\
%		\end{cases}
%		$$
	\end{theorem}
	\begin{proof}
		Generally speaking, any $\s$-field has at least two distinct  $\s$-conjugacy classes: the trivial conjugacy class $\{0\}$ and the conjugacy class containing $1$. If $(K,\s)$  has a nonzero element $a$ which is not conjugate to $1$, then Proposition \ref{(P)Bray_Whaples} implies that there is a skew polynomial of degree two whose roots in $K$ are exactly $1$ and $a$. Therefore, if $K$ has at least three distinct conjugacy classes, then it cannot be  $3$-algebraically closed. To prove the second statement, we assume that $K$ has  exactly two distinct $\s$-conjugacy classes. Let $P(T)\in K[T;\s]$ be a polynomial of degree $>1$ such that $P(0)\neq 0$. Since $(K,\s)$  has only one nontrivial conjugacy classes, we see that the cardinality of the set of roots of $P(T)$ is equal to the cardinality of the projective space $\mathbb{P}(E(P,1))$. Since $(K,\s)$ is  $2$-algebraically closed, we must have $ \dim_F E(P,1)>1$ form which it follows that the cardinality of $\mathbb{P}(E(P,1))$ is at least $\geq 1+|F|$, and we are done. 
		
	\end{proof}
	We remark that a $\s$-field has exactly two distinct $\s$-conjugacy classes iff the map $x\mapsto \s(x)x^{-1}$ is a bijection of $K^*$. As an application, we have the following:
	\begin{corollary}
		Let $n>2$ be a natural number. If $(K,\s)$ is an $n$-algebraically closed $\s$-field which is not $(n+1)$-algebraically closed, then the characteristic $p$ of $K$ is nonzero, and moreover, $n$ is of the form
		$$\frac{p^{mk}-1}{p^k-1},$$
		for some natural numbers $m,k$ with $m>1$. 
	\end{corollary}
	\begin{proof}
		Since $(K,\s)$ is $3$-algebraically closed, it follows from  Theorem \ref{(T)conjugacyclasses3-algebraic} that the fixed field $F$ of $K$ is finite. In particular, the characteristic $p$ of $K$ is nonzero, and $|F|=p^k$ for some natural number $k$. Since  $(K,\s)$ is not $(n+1)$-algebraically closed, there exists a nonconstant polynomial $P(T)\in K[T;\s^n]$ of degree $>1$ such that $P(0)\neq 0$ and $P(T)$ has exactly $n$ distinct roots. Since the number of roots of $P(T)$ is equal to 
		$$\frac{|F|^{\dim_F E(P,1)}-1}{|F|-1}=\frac{p^{k\dim_F E(P,1)}-1}{p^k-1},$$
		the result follows.    		
	\end{proof}
	Another fact regarding algebraically closed $\s$-fields is given in the following proposition.
	\begin{proposition}
		Let $(K,\s)$  be a nontrivial $\s$-field, that is, $\s\neq 1_K$. If $(K,\s)$ is $1$-algebraically closed, then $\s$ is of infinite order. 
	\end{proposition}
	\begin{proof}
		Assume on the contrary that $\s$ is of finite order, say order $n$. Consider the skew polynomial $P_a(T)=T^n-a\in K[T;\s]$ where $a\in K$. An element $x\in K$ is a root of $P_a(T)$ iff 
		$$N_n(x)=\s^{n-1}(x)\cdots \s(x)x=a.$$
		Since $\s^n=1_K$, we see that $N_n(x)=N_{K/F}(x)$ is the norm of $x$ over the fixed field $F$ of $K$. In particular, we must have $a\in F$ if $P_a(T)$ has a root in $K$. Since $(K,\s)$ is $1$-algebraically closed, we conclude that $K=F$, contradicting the assumption that $\s\neq 1_K$.    
	\end{proof}		
	In the next section, we give examples of  $n$-algebraically closed $\s$-fields.
	
\end{subsection}
%%%%%%%%%%%%%%%%%%%%%%%%%%%%%%%%%%%%%%%%%%%%%%%%%%%%%%%%%%%%%%%%%  

\begin{subsection}{Examples of algebraically closed $\s$-fields} 
	
	We begin with a trivial example. Let $K$ be an (ordinary) algebraically closed field. It is easy to see that the trivial $\s$-field $(K,1_K)$ is a $1$-algebraically closed $\s$-field. However, it is not $2$-algebraically closed because the skew polynomial $T^2-2T+1\in K[T;1_K]$ has  exactly one root in $K$.
	
	To give a nontrivial example, we let $K$ be an (ordinary) algebraically closed field of characteristic $p>0$. Let $\s:K\to K$ be the Frobenius homomorphism, that is, $\s(a)=a^p$ for all $a\in K$. Given a natural number $n$, we obtain the $\s$-field $(K,\s^n)$. It is known that the fixed field of $(K,\s^n)$ is the finite field $\mathbb{F}_{p^n}$ with $p^n$ elements.  Working over $(K,\s^n)$, one can easily check that
	$$N_i(a)=a^{\frac{p^{in}-1}{p^n-1}},$$
	for all $a\in K$ and natural numbers $i\geq 1$. Therefore, given a skew polynomial $P(T)=\sum_{i=0}^ma_iT^i\in K[T;\s^n]$ and $a\in K$, we obtain
	$$P(a)=\sum_{i=0}^ma_iN_i(a)=\sum_{i=0}^ma_ia^{\frac{p^{in}-1}{p^n-1}}.$$
	Setting 
	$$f(x)=\sum_{i=0}^ma_ix^{\frac{p^{in}-1}{p^n-1}}\in K[x],$$
	we see that $P(a)=f(a)$ for all $a\in K$. Since $K$ is algebraically closed, we conclude that  $(K,\s^n)$ is $1$-algebraically closed. In fact, more can be said. Assume that $P(0)\neq 0$. Then, $f(x)$ is separable. Since the degree of $f$ is  
	$$\frac{p^{nm}-1}{p^n-1}$$
	where $m$ is the degree of $P(T)$, we see that  the number of distinct roots of the skew polynomial $P(T)$ is
	$$\frac{p^{nm}-1}{p^n-1}.$$
	In particular, we see that $(K,\s^n)$ is $(p^n+1)$-algebraically closed, but not $(p^n+2)$-algebraically closed, since if $\deg P=2$, then $P(T)$ has exactly $p^n+1$ distinct roots.

\end{subsection}
%%%%%%%%%%%%%%%%%%%%%%%%%%%%%%%%%%%%%%%%%%%%%%%%%%%%%%%%%%%%%%%%%  
\begin{subsection}{$c$-Algebraically closed inversive $\s$-fields} 
	In the preceding part, we proved the existence of various types of algebraically closed $\s$-field. Here, we prove the existence of $c$-algebraically closed inversive $\s$-fields. Recall that a $\s$-field   $(K,\s)$ is called inversive if $\s$ is an automorphism. 
	\begin{proposition}\label{(P)c-algebraicallyclosedinversive}
		
		Let $c$ be a cardinal number and $(K,\sigma)$ be $c$-algebraically closed $\s$-field. Then, the inversive closure of $(K,\sigma)$ is  $c$-algebraically closed.
		
	\end{proposition}
	\begin{proof}
		Let $(L,\sigma)$ be the inversive closure of $(K,\sigma)$. Assume that a polynomial $$P(T)=\sum_{i=0}^na_iT^i\in L[T;\s],$$ satisfies $n>1$ and $a_0\neq 0$.
		Since $(L,\sigma)$  is the inversive closure of $(K,\sigma)$, 
		there exists $N\in \mathbb{N}$ such that 
		$\s^N(a_i)\in K$ for all $i=0,1,...,n$. Then, the skew polynomial
		$$Q(T)=\sum_{i=0}^n\s^N(a_i)T^i,$$
		belongs to $ K[T;\s]$.  It is easy to see that if $a\in L$ is a root of $Q(T)$, then $\s^{-N}(a)$ is a root of $P(T)$. Since $Q(T)$ has at least $c$ distinct roots in $K$, we conclude that $P(T)$ has   at least  $c$ distinct roots in $L$. This completes the proof of the proposition. 
	\end{proof}
	As a consequence of this proposition, we have the following result:
	\begin{corollary}\label{(T)embed$c$-algebraicallyclosedinversive}
		
		Let $c$ be a cardinal number. Any $\s$-field $(K,\sigma)$ can be embedded in  a $c$-algebraically closed inversive $\s$-field.
		
	\end{corollary}
	\begin{proof}
		By Theorem \ref{(T)$c$-algebraically closed}, $(K,\s)$ can be embedded in a  $c$-algebraically closed  $\s$-field $(L,\s)$. Then, by Proposition \ref{(P)c-algebraicallyclosedinversive}, the inversive closure of $(L,\s)$ is  $c$-algebraically closed. Since  the inversive closure of $(L,\s)$ contains $(K,\s)$ as a $\s$-subfield, the result follows. 
	\end{proof}

\end{subsection}

%%%%%%%%%%%%%%%%%%%%%%%%%%%%%%%%%%%%%%%%%%%%%%%%%%%%%%%%%%%%%%%%% 
%%%%%%%%%%%%%%%%%%%%%%%%%%%%%%%%%%%%%%%%%%%%%%%%%%%%%%%%%%%%%%%%
%%%%%%%%%%%%%%%%%%%%%%%%%%%%%%%%%%%%%%%%%%%%%%%%%%%%%%%%%%%%%%%%
\end{section} 
%%%%%%%%%%%%%%%%%%%%%%%%%%%%%%%%%%%%%%%%%%%%%%%%%%%%%%%%%%%%%%%%
%%%%%%%%%%%%%%%%%%%%%%%%%%%%%%%%%%%%%%%%%%%%%%%%%%%%%%%%%%%%%%%%
%%%%%%%%%%%%%%%%%%%%%%%%%%%%%%%%%%%%%%%%%%%%%%%%%%%%%%%%%%%%%%%%
%%%%%%%%%%%%%%%%%%%%%%%%%%%%%%%%%%%%%%%%%%%%%%%%%%%%%%%%%%%%%%%%
%%%%%%%%%%%%%%%%%%%%%%%%%%%%%%%%%%%%%%%%%%%%%%%%%%%%%%%%%%%%%%%%
%%%%%%%%%%%%%%%%%%%%%%%%%%%%%%%%%%%%%%%%%%%%%%%%%%%%%%%%%%%%%%%%
\begin{section}{Partial fraction decomposition}
	
	In this section, we present a version of a partial fraction decomposition for  the field of fractions of the ring of skew polynomials. We begin with some standard definitions and facts regarding factorization in  (not necessarily commutative) principal ideal domains, see \cite{cohn1963noncommutative,cohn1973unique} for more details. Let $R$ be a (not necessarily commutative) principal ideal domain (or PID for short). A nonunit $a\in R$ is called irreducible if $a=bc$, where $b,c\in R$, implies that $b$ or $c$ is a unit in $R$.
	Elements $a,b\in R$ are called similar if the left modules $R/Ra$ and $R/Rb$ are isomorphic, or equivalently, the right modules $R/aR$ and $R/bR$ are isomorphic. Elements $a,b\in R$ are called left (resp., right) co-prime if $aR+bR=R$ (reps., $Ra+Rb=R$). They are called co-prime if they are both left and right co-prime. An element $a\in R$ is called right decomposable if $a=b_1c_1=b_2c_2$ for some nonunits $b_1,b_2,c_1,c_2\in R$ such that $Rc_1+Rc_2=R$. An element is called right indecomposable if it is not right decomposable. The notion of left decomposable (indecomposable) is defined similarly.
	 It is known that any PID $R$ is a UFD, that is, every nonunit $a\in R$ can be written as
	$$a=a_1a_2\cdots a_n,$$ where $a_1,...,a_n$ are irreducible, and moreover, if $a=b_1\cdots b_m$, where $b_1,...,b_m$ are irreducible, then $m=n$, and there is a permutation $\tau\in S_n$ such that each $b_i$ is similar to $a_{\tau(i)}$.

	%%%%%%%%%%%%%%%%%%%%%%%%%%%%%%%%%%%%%%%%%%%%%%%%%%%%%%%%%%%%%%%%
	%%%%%%%%%%%%%%%%%%%%%%%%%%%%%%%%%%%%%%%%%%%%%%%%%%%%%%%%%%%%%%%%
	\begin{subsection}{Partial fraction decomposition in a PID}
		
		In this subsection, $R$ is a fixed PID. It is known that $R$ is an Ore domain, hence has a field of fractions $F=Frac(R)$. A fraction $ab^{-1}\in F$, where $a,b\in R$, is called \textit{reduced} if $a$ and $b$ are right co-prime, that is, $Ra+Rb=R$. 
		\begin{lemma}
			Any $x\in F$ has at least  one reduced fraction $x=ab^{-1}$ which is unique up to multiplication by units. More precisely, if $ab^{-1}=a_1b_1^{-1}$ are two reduced fractions, then $a_1=au$ and $b_1=bu$ for some unit $u\in R$.  More generally, if $ab^{-1}=a_1b_1^{-1}$  and $ab^{-1}$ is reduced, then there exists $r\in R$ such that  $a_1=ar$ and $b_1=br$
		\end{lemma}
		\begin{proof}
			Let $x=ab^{-1}$ be an arbitrary element of $F$. We have $Ra+Rb=Rc$ for some $c\in R$ because $R$ is a PID. It follows that 
			$a=a_1c$ and $b=b_1c$ for some $a_1,b_1\in R$.
			Therefore, $Ra_1+Rb_1=R$ which implies that $x=a_1b_1^{-1}$ is a reduced fraction. This proves the existence part. 
			Now, let $ab^{-1}=a_1b_1^{-1}$, where $ab^{-1}$ is reduced, that is,
			$xa+yb=1$
			for some $x,y\in R$. Since $R$ is an Ore domain, $bR\cap b_1R\neq\{0\}$ from which it follows that there exist $c,c_1\in R^*$ such that $bc=b_1c_1$. We have
			$$ ab^{-1}=a_1b_1^{-1}\implies (ab^{-1})(bc)=(a_1b_1^{-1})(b_1c_1)\implies ac=a_1c_1.$$
			Therefore, we obtain
			$$c=(xa+yb)c=xac+ybc=xa_1c_1+yb_1c_1=(xa_1+yb_1)c_1.$$ Setting  $r=xa_1+yb_1$,  we obtain
			$$a_1c_1=ac=arc_1\implies a_1 = a r.$$
			In a similar way, one can show that $b_1=br$. This completes the proof of the lemma.   
		\end{proof}
		The following result can be regarded as a generalization of the method of partial fraction decomposition,  see \cite{lang2012algebra} for the case of a commutative prinicipal ideal domain.  
		\begin{theorem}\label{(T)Partialfractiondecomposition}
			Every fraction of $R$ can be be written as a sum of fractions of the form $ab^{-1}$, where $a,b\in R$  with $0\neq b$ being right indecomposable. 
		\end{theorem}
		\begin{proof}	
			Let $S$ be the set of all fractions that cannot be written in the stated form. We need to show that $S$ is empty.  Assume, on the contrary, that $S$ is nonempty. Consider the set $I$ of all ideals of the form $Rb$ where   $ab^{-1}\in S$ for some nonunits $a\in R$ and $b\in R^*$. Since $R$ is a PID, $I$ has a maximal element $Rb_0$ with respect to inclusion. Then $ab_0^{-1}\in S$ for some $a\in R$. Note that $b_0$ cannot be right indecomposable since otherwise $ab_0^{-1}\notin S$.   So, $b_0$ must be right decomposable, that is, we have $b_0=b_1c_1=b_2c_2$ for some $b_1,c_1,b_2,c_2\in R$ such that $Rc_1+Rc_2=R$. In particular, there exist $a_1,a_2\in R$ such that
			$a_1c_1+a_2c_2=a$. Moreover, $b_0R\subsetneq b_1R$ and $b_0R\subsetneq b_2R$ since $c_1$ and $c_2$ are non-units.  Now, we have
			$$a=a_1c_1+a_2c_2\implies ab_0^{-1}=a_1b_1^{-1}+a_2b_2^{-1}.$$
			By the maximality of $Rb_0$, neither $a_1b_1^{-1}$ nor $a_2b_2^{-1}$ can belong to $ S$.  Therefore, both $a_1b_1^{-1}$ and $a_2b_2^{-1}$ can be written in the desired form. Consequently, $ab^{-1}$ which is the sum of these two fractions can be written in the desired form, a contradiction. This complete the proof. 
		\end{proof}
		
	\end{subsection}
	
	%%%%%%%%%%%%%%%%%%%%%%%%%%%%%%%%%%%%%%%%%%%%%%%%%%%%%%%%%%%%%%%%
	\begin{subsection}{Partial fraction decomposition in $K(T;\sigma)$}
		Let $(K,\s)$ be an inversive $\s$-field. It is known that $K[T;\sigma]$ is a PID. Its field of fractions is denoted by $K(T;\s)$. In particular, we can apply Theorem \ref{(T)Partialfractiondecomposition} to $K[T;\sigma]$. In fact, we have the following stronger result in the case of $K(T;\s)$.
		\begin{theorem}\label{(T)Partialfractiondecompositioninsigmafields}
			Every fraction $P(T)Q(T)^{-1}\in K(T;\s)$ can be written as
			$$P(T)Q(T)^{-1}=P_0(T)+\sum_{i=1}^m P_i(T)Q_i(T)^{-1},$$
			where $P_0(T),P_1(T),Q_1(T),...,P_n(T),Q_n(T)\in K[T;\sigma]$, $\deg P_i<\deg Q_i$ for every $i$, each $Q_i(T)$ is right indecomposable, and each fraction $P_i(T)Q_i(T)^{-1}$ is reduced. 
		\end{theorem}
		\begin{proof}	
			Let $S$ be the set of all partial fraction decompositions
				$$P(T)Q(T)^{-1}=P_0(T)+\sum_{i=1}^m P_i(T)Q_i(T)^{-1},$$
			where $P_0(T),P_1(T),Q_1(T),...,P_m(T),Q_m(T)\in K[T;\sigma]$,  and each $Q_i(T)$ is right indecomposable of degree $>0$.
			By Theorem \ref{(T)Partialfractiondecomposition}, the set $S$ is nonempty. Therefore, we can choose $$P(T)Q(T)^{-1}=P_0(T)+\sum_{i=1}^m P_i(T)Q_i(T)^{-1}$$
			in $S$ for which $\sum_{i=1}^m\deg P_i$ is as small as possible. I claim that this decomposition satisfies the desired conditions. Assume, on the contrary, that $\deg P_i\geq \deg Q_i$ for some $i$, say $i=1$. 
			Using the Euclidean division algorithm in $K[T;\sigma]$, we write  
			$$P_1(T)=P'(T)Q_1(T)+R(T),$$
			where $P'(T),R(T)\in K[T;\sigma]$ and $\deg R<\deg Q.$ 
			Then, the decomposition 
			$$P(T)Q(T)^{-1}=\left( P_0(T)+P'(T)\right) +\left( R(T)Q_1(T)^{-1}+\sum_{i=2}^m P_i(T)Q_i(T)^{-1}\right) ,$$ belongs to $S$, which is a contradiction since 
			$$\deg R+\sum_{i=2}^m\deg P_i<\deg Q+\sum_{i=2}^m\deg P_i\leq \sum_{i=1}^m\deg P_i.$$
			 This completes the proof of the claim. If some $P_i(T)Q_i(T)^{-1}$ is not reduced, we can replace $P_i(T)Q_i(T)^{-1}$ by a reduced fraction whose numerator has a smaller degree than the degree of $P_i(T)$, contradicting the choice of the partial fraction decomposition. This completes the proof. 			 
		\end{proof}	

		Next, we show that a stronger result is valid in the case of $2$-algebraically closed $\s$-fields.

		\begin{theorem}\label{(T)Partialfractiondecompositioninalgebraicallyclosed}
			Let $(K,\s)$ be a $2$-algebraically closed inversive $\s$-field. Any fraction $P(T)Q(T)^{-1}\in K(T;\s)$  has a partial fraction decomposition of the form
			$$P(T)Q(T)^{-1}=P_0(T)+\sum_{i=1}^m a_iT^{-i}+\sum_{i=1}^n b_i(1-c_iT)^{-1},$$
			where $P_0(T)\in K[T;\sigma]$, $a_1,...,a_m\in K$, $b_1,...,b_n,c_1,...,c_n\in K^*$,  and  $c_1,...,c_n$ are (right) P-independent.  Moreover,  if 
			$$P(T)Q(T)^{-1}=P_1(T)+\sum_{i=1}^{m_1} a'_iT^{-i}+\sum_{i=1}^{n_1} b'_i(1-c'_iT)^{-1},$$
			is another such decomposition, then we have
			$P_0(T)=P_1(T)$, $m=m_1$, and  $a_i=a'_i$ for all $i$. 
		\end{theorem}
		\begin{proof}	
			First, we show that a nonconstant polynomial $P(T)\in K[T;\sigma]$ is right indecomposable iff $P(T)=aT^m$ for some $m>0$ and $a\in K^*$, or $P(T)=a(1-cT)$ for some $a,c\in K^*$. To prove the "only if" part of this statement, assume that $P(T)\in K[T;\sigma]$ is right indecomposable. There is $m\geq 0$ and $Q(T)\in K[T;\sigma]$ with a nonzero constant term such that $P(T)=Q(T)T^m$. Since $(K,\s)$ is inversive, there exists $Q_1(T)\in K[T;\sigma]$ with a nonzero constant term such that $$P(T)=Q(T)T^m=T^mQ_1(T).$$
			Since $Q_1(T)$ and $T^m$ are right co-prime, $m$ must be zero or $Q(T)$ must be a constant. In the latter case, we have $P(T)=aT^m$ for some $a\in K^*$. In the former case, $Q(T)$ must be a linear skew polynomial, since otherwise the fact that $(K,\s)$ is $2$-algebraically closed implies that $Q(T)$ has at least two distinct roots, say $a_1\neq a_2$. It follows that 
			$$Q(T)=R_1(T)(T-a_1)=R_2(T)(T-a_2),$$    
			for some polynomials $R_1(T),R_2(T)\in K[T;\sigma]$. Since $T-a_1$ and $T-a_2$ are right co-prime, $Q(T)$ must be right decomposable, a contradiction. This completes the proof of the "only if" part.  Conversely, it is clear that any linear skew polynomial is right indecomposable. Since $T^m=P(T)Q(T)$ implies that $P(T)=aT^n$ and  $Q(T)=(1/a)T^{m-n}$ for some $0\neq n\neq m$, we see that $T^m$ is right indecomposable. This completes the proof of the statement. Now, we can prove the theorem: It follows from the above statement and  Theorem \ref{(T)Partialfractiondecompositioninsigmafields} that any $P(T)Q(T)^{-1}\in K(T;\s)$  has a partial fraction decomposition of the form
			$$P(T)Q(T)^{-1}=P_0(T)+\sum_{i=1}^m a_iT^{-i}+\sum_{i=1}^n b_i(1-c_iT)^{-1},$$
			where $P_0(T)\in K[T;\sigma]$, $a_1,...,a_m,b_1,...,b_n\in K$, and the elements $c_1,...,c_n$ belong to $K^*$. Using Proposition \ref{(P)FormulationofVandermonde}, it is now easy to show that there is such a partial fraction decomposition which, in addition, satisfies the requirement that $c_1,...,c_n$ be P-independent. 
			
			 To prove the uniqueness part, we first note that the summand $\sum_{i=1}^m a_iT^{-i}$ does not belong to $K[[T;\s]]^*$ while the other summands belong to  $K[[T;\s]]$, from which the equality $\sum_{i=1}^m a_iT^{-i}=\sum_{i=1}^{m_1} a'_iT^{-i}$ follows. Finally, using the degree function on $K(T;\s)$, one can easily show that  $P_1(T)$ must be equal to $P_0(T)$. 
	\end{proof}	
	
	\end{subsection}
	%%%%%%%%%%%%%%%%%%%%%%%%%%%%%%%%%%%%%%%%%%%%%%%%%%%%%%%%%%%%%%%%
	%%%%%%%%%%%%%%%%%%%%%%%%%%%%%%%%%%%%%%%%%%%%%%%%%%%%%%%%%%%%%%%%
%%%%%%%%%%%%%%%%%%%%%%%%%%%%%%%%%%%%%%%%%%%%%%%%%%%%%%%%%%%%%%%%
\begin{subsection}{The twisted Hadamard algebra}
	We conclude this paper with an application for which we need to recall the definition of the twisted Hadamard algebra over a $\s$-field. For more details, see \cite{aryapoor2022rational}. Let $(K,\s)$ be an inversive $\s$-field. An element in $K[[T,\s]]$ is called a rational twisted series if it belongs to the subring $K(T;\s)$. A rational twisted series is called regular if its degree is negative. The set of all regular rational series over $(K,\s)$ is denoted by $H(T;\s)$. In \cite{aryapoor2022rational}, it is shown that the set $H(T;\s)$ equipped with the addition operation inherited from $K[[T,\s]]$ and the Hadamard multiplication $\odot$, that is, 
	$$\left( \sum_{i=0}^\infty a_i T^i\right) \odot \left(\sum_{i=0}^\infty b_i T^i\right) =
	\sum_{i=0}^\infty a_i b_i T^i,$$	
	is a $K$-algebra, called the twisted Hadamard algebra over $(K,\s)$. Here, we give a characterization of the  twisted Hadamard algebra in the case when $(K,\s)$ is 2-algebraically closed. For a characterization of the ordinary Hadamard algebra in terms of exponential polynomials, see \cite{berstel1988rational}. In the following theorem, we use the $K$-algebra $N_{K^*}$  introduced in Subsection \ref{(S)preliminaries}. 
	\begin{theorem}\label{(T)Hadamard}
		There exists a unique $K$-algebra homomorphism 
		$$\alpha:N_{K^*}\to H(T;\s),$$
		which satisfies $\alpha(N(a))=(1-aT)^{-1}$ for all $a\in K^*$. Furthermore, $\alpha$ is injective, and, if $K$ is a  2-algebraically closed inversive $\s$-field,  $\alpha$ is an isomorphism of $K$-algebras. 
	\end{theorem}
	\begin{proof}
		The vector space $K^\infty $, as defined in Subsection \ref{(S)preliminaries}, is a $K$-algebra under pointwise multiplication.  Consider the map $\beta:K^\infty\to K[[T;\s]]$ defined by 
		$$\beta(a_0,a_1,a_2,...)=\sum_{i=0}^\infty a_i T^i.$$
		It is easy to see that $\beta$ is a $K$-algebra homomorphism where $K[[T;\s]]$ is equipped with the Hadamard product. Let $\alpha:N_{K^*}\to K[[T;\s]]$ be the restriction of $\beta$ to $N_{K^*}$.  Note that $(1-aT)^{-1}\in H(T;\s)$, for all $a\in K^*$. Since  $N_{K^*}$, as a vector space, is generated by $N(a)$'s, and $H(T;\s)$ is closed under the Hadamard product, we see that the image of $\alpha$ is contained in $H(T;\s)$. The uniqueness of $\alpha$ follows from the fact that the elements $N(a)$, $a\in K$, generate $N_{K^*}$ as a vector space over $K$. Clearly, $\alpha$ is injective. The fact that $\alpha$ is surjective when $K$ is a $2$-algebraically closed inversive $\s$-field follows from Theorem  \ref{(T)Partialfractiondecompositioninalgebraicallyclosed}.      
		
	\end{proof}
	As a consequence of this result, we have the following:
	\begin{corollary}
		Let $\sum_{i=0}^{\infty}a_iT^i$ be a rational twisted series over a $2$-algebraically closed inversive $\s$-field $K$. Then, there exist element  $b_1,...,b_n,c_1,...,c_n\in K^*$ such that 
		$$a_i=\sum_{j=1}^n b_jN_i(c_j),$$
		for all large enough $i$. 		
	\end{corollary}
	\begin{proof}
		Since every rational twisted series is a sum of a skew polynomial and a regular series, the result follows from Theorem \ref{(T)Hadamard}. 
	\end{proof}
	
\end{subsection}
%%%%%%%%%%%%%%%%%%%%%%%%%%%%%%%%%%%%%%%%%%%%%%%%%%%%%%%%%%%%%%%%
%%%%%%%%%%%%%%%%%%%%%%%%%%%%%%%%%%%%%%%%%%%%%%%%%%%%%%%%%%%%%%%%
\textbf{Acknowledgements}\\
I would like to thank A. Leroy for
his valuable suggestions and comments which led to an improvement of the paper.
\end{section}

\bibliographystyle{plain}
\bibliography{SRbiblan}

\begin{thebibliography}{10}

\bibitem{aryapoor2022rational}
Masood Aryapoor.
\newblock Rational twisted series.
\newblock {\em arXiv preprint arXiv:2202.11131}, 2022.

\bibitem{berstel1988rational}
Jean Berstel and Christophe Reutenauer.
\newblock {\em Rational series and their languages}, volume~12.
\newblock Springer-Verlag, 1988.

\bibitem{bolotnikov2020lagrange}
Vladimir Bolotnikov.
\newblock Lagrange interpolation over division rings.
\newblock {\em Communications in Algebra}, 48(9):4065--4084, 2020.

\bibitem{braywhaples1983polynomials}
Una Bray and George Whaples.
\newblock Polynomials with coefficients from a division ring.
\newblock {\em Canadian Journal of Mathematics}, 35(3):509--515, 1983.

\bibitem{chuang2013automorphisms}
Chen-Lian Chuang.
\newblock Automorphisms of ore extensions.
\newblock {\em Israel Journal of Mathematics}, 197(1):437--452, 2013.

\bibitem{cohn1963noncommutative}
Paul~M Cohn.
\newblock Noncommutative unique factorization domains.
\newblock {\em Transactions of the American Mathematical Society},
  109(2):313--331, 1963.

\bibitem{cohn1995encyclopedia}
Paul~M Cohn.
\newblock {\em Skew Fields: Theory of General Division Rings}.
\newblock Encyclopedia of Mathematics and its Applications, 57. Cambridge
  University Press, Cambridge, 1995.

\bibitem{cohn1973unique}
Paul~Moritz Cohn.
\newblock Unique factorization domains.
\newblock {\em The American Mathematical Monthly}, 80(1):1--18, 1973.

\bibitem{cohn1975presentations}
PM~Cohn.
\newblock Presentations of skew fields. i. existentially closed skew fields and
  the nullstellensatz.
\newblock In {\em Mathematical Proceedings of the Cambridge Philosophical
  Society}, volume~77, pages 7--19. Cambridge University Press, 1975.

\bibitem{delenclos2007noncommutative}
Jonathan Delenclos and Andr{\'e} Leroy.
\newblock Noncommutative symmetric functions and w-polynomials.
\newblock {\em Journal of Algebra and its Applications}, 6(05):815--837, 2007.

\bibitem{Gelfandetal1995study}
I.M. Gelfand, D.~Krob, A.~Lascoux, B.~Leclerc, V.S. Retakh, and J.Y. Thibon.
\newblock Noncommutative symmetric functions.
\newblock {\em Advances in Mathematics}, 112(2):218--348, 1995.

\bibitem{gelfand1996noncommutative}
Israel Gelfand and Vladimir Retakh.
\newblock Noncommutative vieta theorem and symmetric functions.
\newblock In {\em The Gelfand Mathematical Seminars, 1993--1995}, pages
  93--100. Springer, 1996.

\bibitem{lam1991first}
Tsit-Yuen Lam.
\newblock {\em A first course in noncommutative rings}, volume 131.
\newblock Springer, 1991.

\bibitem{LamLeroy1988}
Tsit-Yuen Lam and Andr{\'e} Leroy.
\newblock Vandermonde and wronskian matrices over division rings.
\newblock {\em Journal of Algebra}, 119(2):308--336, 1988.

\bibitem{lam2004wedderburn}
Tsit~Yuen Lam and Andr{\'e} Leroy.
\newblock Wedderburn polynomials over division rings, i.
\newblock {\em Journal of Pure and Applied Algebra}, 186(1):43--76, 2004.

\bibitem{lam1985general}
T.Y. Lam.
\newblock A general theory of vandermonde matrices.
\newblock {\em Exposition. Math.}, 4(3):193--215, 1986.

\bibitem{lam1994hilbert}
TY~Lam and A~Leroy.
\newblock Hilbert 90 theorems over division rings.
\newblock {\em Transactions of the American Mathematical Society},
  345(2):595--622, 1994.

\bibitem{lamleroy1988algebraic}
TY~Lam and Andr{\'e} Leroy.
\newblock Algebraic conjugacy classes and skew polynomial rings.
\newblock In {\em Perspectives in ring theory}, pages 153--203. Springer, 1988.

\bibitem{lang2012algebra}
Serge Lang.
\newblock {\em Algebra}, volume 211.
\newblock Springer Science \& Business Media, 2012.

\bibitem{levin2008difference}
Alexander Levin.
\newblock {\em Difference algebra}, volume~8.
\newblock Springer Science \& Business Media, 2008.

\bibitem{neumann1943adjunction}
Bernhard~Hermann Neumann.
\newblock Adjunction of elements to groups.
\newblock {\em Journal of the London Mathematical Society}, 1(1):4--11, 1943.

\bibitem{niven1941equations}
Ivan Niven.
\newblock Equations in quaternions.
\newblock {\em The American Mathematical Monthly}, 48(10):654--661, 1941.

\bibitem{ore1933theory}
Oystein Ore.
\newblock Theory of non-commutative polynomials.
\newblock {\em Annals of mathematics}, pages 480--508, 1933.

\bibitem{robinson1971notion}
Abraham Robinson.
\newblock On the notion of algebraic closedness for noncommutative groups and
  fields.
\newblock {\em The Journal of Symbolic Logic}, 36(3):441--444, 1971.

\bibitem{scott1951algebraically}
William~R Scott.
\newblock Algebraically closed groups.
\newblock {\em Proceedings of the American Mathematical Society},
  2(1):118--121, 1951.

\bibitem{treur1989separate}
Jan Treur.
\newblock Separate zeros and galois extensions of skew fields.
\newblock {\em Journal of Algebra}, 120(2):392--405, 1989.

\end{thebibliography}

 \end{document}